\documentclass[10pt]{amsart}

\usepackage{CJKutf8}

\usepackage{amsmath}
\usepackage{amssymb}
\usepackage{amsthm} 
\usepackage{graphicx}
\usepackage{color}
\usepackage{scalerel}
\usepackage[dvipsnames]{xcolor}
\usepackage{stmaryrd}
\usepackage{soul}
\usepackage{comment}
\usepackage{mathabx}
\usepackage{enumerate}

\usepackage{amsfonts}
\usepackage{amssymb}
\usepackage[hidelinks]{hyperref}
\usepackage[capitalize]{cleveref}

\usepackage{nicefrac}
\usepackage{xfrac}

\usepackage[english]{babel}
\usepackage[autostyle]{csquotes}
\usepackage[shortlabels]{enumitem}
\usepackage[backend=bibtex, style=alphabetic, sorting=nyt, doi=false, isbn=false, url=false]{biblatex}
\renewbibmacro*{in:}{}  

\usepackage{tipa}
\usepackage[normalem]{ulem}

\usepackage{mathtools}

\DeclareMathOperator*{\forkindep}{\raise0.2ex\hbox{\ooalign{\hidewidth$\vert$\hidewidth\cr\raise-0.9ex\hbox{$\smile$}}}}

\newcommand{\tp}{\operatorname{tp}}

\DeclareTextCommand{\DZ}{OT2}{D2}

\newcommand{\im}{\mathrm{im}}

\newtheorem*{claim-star}{Claim}
\newtheorem*{theorem-non}{Theorem}
\newtheorem{theorem}{Theorem}[section] 
\newtheorem{lemma}[theorem]{Lemma}
\newtheorem{notation}[theorem]{Notation}

\newtheorem{prop-def}[theorem]{Proposition-Definition}
\newtheorem{corollary}[theorem]{Corollary}
\newtheorem{fact}[theorem]{Fact}
\newtheorem{fact-eh}[theorem]{Fact(?)}

\newtheorem{question}[theorem]{Question}
\newtheorem{proposition}[theorem]{Proposition}
\newtheorem{proposition-eh}[theorem]{Proposition(?)}
\newtheorem*{theorem-star}{Theorem}
\newtheorem*{conjecture-star}{Conjecture}
\newtheorem*{lemma-star}{Lemma}

\theoremstyle{definition}
\newtheorem{definition}[theorem]{Definition}
\newtheorem{example}[theorem]{Example}

\newtheorem{remark}[theorem]{Remark}

\theoremstyle{remark}

\newtheorem{observation}[theorem]{Observation}

\newcommand{\supp}{\mathrm{supp}}

\newcommand{\Th}{\mathrm{Th}}
\newcommand{\cnv}{\mathrm{cnv}}
\newcommand{\grp}{\mathrm{grp}}

\newenvironment{claimproof}[1][\proofname]
               {
                 \proof[#1]
                 
               }
               {
                 \endproof
               }

\allowdisplaybreaks 

\title{Model theory of convolution algebras}

\author[A. Berenstein]{Alexander Berenstein}
\address{Universidad de los Andes,
Cra 1 No 18A-12, Edificio H, Bogot\'{a}, Colombia}
\urladdr{www.matematicas.uniandes.edu.co/\textasciitilde aberenst}

\author[K. Gannon]{Kyle Gannon}
\address{Beijing International Center for Mathematical Research (BICMR) \\ Peking University \\ Beijing, China.}
\email{kgannon@bicmr.pku.edu.cn}
\urladdr{http://faculty.bicmr.pku.edu.cn/\textasciitilde kyle/}

\author[S. Song]{Shichang Song}
\address{School of Mathematics and Statistics\\Beijing Jiaotong University\\3 Shangyuancun Haidian District\\Beijing, China}
 \urladdr{https://faculty.bjtu.edu.cn/8629/}


\begin{document}

\begin{abstract}
This paper deals with the model theory of convolution algebras $(L^1(G),*)$ for locally compact groups $G$, seen as Banach lattices equipped with the convolution product $*$. We first prove transfer principles for elementary equivalence and elementary embeddings when the underlying group $G$ is discrete, namely, $(\ell^1(G),*) \equiv (\ell^1(H),*)$ implies $G \equiv H$, while the converse holds when $G$ and $H$ are $\omega$-saturated (likewise for elementary substructures). Without $\omega$-saturation, the proof of the converse fails. Although pure Banach lattices are model-theoretically tame, our results imply that adding convolution yields wild behavior. For example, we prove that if $G$ is any locally compact, non-discrete group, then the formula $d(x,x*y)$ is unstable with respect to $\Th(L^1(G),*)$. Moreover, we show that if $G$ is discrete and contains a particular configuration of amenable subgroups, then the formula $d(x*y,z)\dotdiv\frac{1}{2}$ witnesses $\mathrm{TP}_2$ with respect to $\Th(\ell^1(G),*)$. As a consequence, if $G$ contains an infinite abelian subgroup, then $\Th(\ell^{1}(G),*)$ has $\mathrm{TP}_{2}$. We prove similar results in the locally compact non-discrete setting using the notion of an \emph{approximate identity}. Finally, we prove a `continuous-by-discrete' approximation theorem. Namely, convolution algebras of connected abelian Lie groups admit metric embeddings into ultraproducts of convolution algebras over finite abelian groups.

\end{abstract}

\maketitle

\section{Introduction}

The theory of the Banach lattice $(L^1(\mathbb{R}); +, 0, \wedge, \vee, \lVert \cdot \rVert)$ is model-theoretically tame: its models are Banach lattices over atomless measure spaces, it admits quantifier elimination \cite[Example 13.18]{HeIo}, is $\aleph_0$-stable, and has an explicit characterization of forking independence \cite{BYBHLp}. Despite this tameness, even the pure lattice structure exhibits some pathological behavior: it fails to be $\aleph_0$-categorical after adding a constant, fails to be strongly finitely based \cite{yaacov2014almost}, and is multidimensional\footnote{Other expansions of Banach $L_p$-lattices have been studied in the literature, mostly by adding automorphisms. For example, expansions by a generic automorphism/an amenable group of generic automorphisms were studied in \cite{ScielzoZ,ScielzoAmenable}.}. However, many Banach spaces of analytic interest carry additional algebraic structure. For any locally compact group $G$, the space $L^1(G)$ admits a natural convolution product after fixing a Haar measure. In this article, we study the model theory of convolution algebras $(L^1(G); +, 0, \wedge, \vee, \lVert \cdot \rVert, *)$. Our main thesis is that adding convolution drastically alters the model-theoretic complexity: whereas the underlying lattice structure is well-understood and tame, the resulting algebra is often wild, definably encoding the combinatorial complexity of the original group's amenable subgroups. Nevertheless, this complexity remains tractable in several ways, and we obtain a number of positive results. 

We write $(L^1(G),*)$ for the convolution algebra of a locally compact group $G$, and $(\ell^1(G),*)$ in the discrete case. Our main results are as follows. For discrete groups $G$ and $H$, we first establish a transfer principle between elementary equivalence of convolution algebras and elementary equivalence of the underlying groups. Specifically, if $(\ell^{1}(G),*) \equiv (\ell^{1}(H),*)$, then $G \equiv H$ as pure groups (Corollary \ref{interpretgroup}). This fact follows essentially from the observation that an isomorphic copy of $G$ is definable in $(\ell^{1}(G),*)$ via identify $G$ with the atoms. Conversely, if $G \equiv H$ and both are $\omega$-saturated, then $(\ell^{1}(G),*) \equiv (\ell^{1}(H),*)$ (Theorem \ref{thm:elementaryequiv}). This direction relies on Tarski-Vaught and an automorphism-lifting argument. Analogous results hold for elementary substructures. Following observations by Keller \cite{Keller1972}, we show that the $\omega$-saturation assumption cannot be omitted in the converse statement (Theorem \ref{thm:failure}). Indeed, \emph{amenability} is an elementary property in the language of convolution algebras and thus any group which is amenable but not uniformly amenable can be used to construct a counterexample to the converse. This result has a similar flavor to the main result in \cite{Goldbring2024}, which presents elementarily equivalent countable ICC groups whose von Neumann algebras are not elementarily equivalent, but with simpler techniques and in a simpler setting.

We next consider convolution algebras through the lens of classification theory. As stated earlier, Banach lattices admit a relatively tame model-theoretic classification as metric structures. However, equipping $L^1(G)$ with convolution often introduces wild behavior. We first prove that if $G$ is a locally compact, non-compact amenable group, then the formula $d(x * y, y)$ is unstable with respect to $\Th(L^{1}(G),*)$ (Discrete case, Proposition \ref{prop:unstable}; general case, Proposition \ref{prop:cont_amen}). Moreover, we further show that $\Th(L^{1}(G),*)$ is unstable whenever $G$ is non-discrete, leveraging the existence of a bounded approximate identity (Proposition \ref{notstab-compact}). We then prove that if a group $G$ contains either an element of infinite order or elements of arbitrarily large finite order, then its associated convolution algebra has $\mathrm{TP}_2$ (Theorem \ref{thm:ContIP}). Moreover, we establish that a certain natural configuration of amenable subgroups implies the convolution algebra admits $\mathrm{TP}_2$ (Theorem \ref{thm:general}). In the language of the convolution algebra, one can approximately uniformly encode the relation $a \in bH$, where $a,b \in G$ and $H$ is an amenable subgroup of $G$. Leveraging this, one can find elements in elementary extensions that behave like \emph{invariant means} on amenable subgroups relative to the diagonal embedding. Consequently, if $G$ contains an infinite abelian subgroup, then $(\ell^{1}(G),*)$ has $\mathrm{TP}_2$ (Corollary \ref{cor:tp2}). To illustrate, our results imply that $\ell^{1}(\mathbb{Z})$, $\ell^{1}\bigl(\bigoplus_{i < \omega} \mathbb{Z}/2\mathbb{Z}\bigr)$, and $L^{1}(\mathbb{R})$ all have $\mathrm{TP}_2$. In each of these cases, the property is witnessed by the formula $d(x \ast y, z) \mathbin{\dotdiv} \frac{1}{2}$. These results answer a question of Khanaki \cite[Question 11.3]{Khanaki}, who asked whether $(\ell^{1}(\mathbb{Z}),*)$ has $\mathrm{IP}$; indeed it does, witnessed by $d(x * y, y)$ (Corollary \ref{cor:tp2}). It remains open whether or not $(\ell^{1}(M),*)$ has $\mathrm{TP}_{2}$ (or even $\mathrm{IP}$), where $M$ is a Tarski monster (Question \ref{Question:Tarski}).

Finally, we prove an embedding theorem: for any connected abelian Lie group, its 
convolution algebra metrically embeds into a metric ultraproduct of convolution algebras 
over finite abelian groups (Theorem \ref{theorem:embeds}). This gives an instance where continuous convolution algebras arise naturally as limits of their discrete counterparts.

\section*{Acknowledgments}
We would like to thank Jinhe Ye for pointing out to us that every infinite solvable group contains an infinite abelian subgroup. We also like to thank Isaac Goldbring and Boris Zilber for helpful discussions related to this project. Additionally, the first named author would like to thank the organizers of the Beijing Model Theory Conference of 2024 where the first ideas behind this project were discussed.

This research was supported in part by CIRED-Sanford 2026. Gannon was supported in part by the Fundamental Research Funds for the Central Universities, Peking University, grant no.\ 7100604835 and by the National Natural Science Fund of China, grant no.\ 12501001. Song was supported in part by the Fundamental Research Funds for the Central Universities, Beijing Jiaotong University, grant no.\ 2025JBZX010.

Over the course of this project, the capabilities of AI models have changed significantly. Towards the beginning of this project, the second named author reports that multiple attempts to employ free AI models for experimental computations in $\ell^{1}(\mathbb{Z})$ proved somewhat unproductive. In recent months, the subscription based models have developed more serious capabilities, though we did not use these models on this project. However, Deepseek was helpful at directing the authors toward several references in the literature which were ultimately helpful. 

\section{Preliminaries}

Our notation is relatively standard regarding both logic and analysis. We refer the reader to \cite{BBHU2008} for basic background on continuous logic. We establish the following basic conventions: 

\begin{enumerate}
    \item If $r,s,\epsilon$ are real numbers with $\epsilon >0$, we write $r \approx_{\epsilon} s$ to mean that $|r - s| < \epsilon$. 
    \item The letters $G$ and $H$ will always denote groups, possibly with or without non-trivial topological structure (which will be established in the appropriate sections and statements). 
    \item  Formally, we let $\mathcal{L}_{\grp} = \{\cdot\}$ be the language of pure groups with only the product operation as a basic symbol. Both the identity and inverse map are definable in this language and so no information is lost. We use this language to simplify some of our arguments.
    \item A topological group is a group with a topology such that both multiplication and inversion are continuous. 
    \item If $G$ is a discrete group, we let $(\ell^{1}(G);+,0,\wedge,\vee,\lVert \ \ \rVert,*)$ denote the Banach lattice equipped with the convolution operation, $*$. We refer to this algebra as a \emph{convolution algebra} and will often abbreviate the signature by simply writing $(\ell^{1}(G),*)$. Unless specified otherwise, the convolution product is computed with respect to the standard counting measure on $G$.
    \item More generally, if $(G,\mu)$ is a locally compact group with a fixed left Haar measure $\mu$, we let $(L_{\mu}^{1}(G);+,0,\wedge,\vee,\lVert \ \ \rVert,*)$ denote the $L^{1}$-Banach lattice equipped with the convolution operation (computed relative to $\mu$). We refer to this algebra as a \emph{convolution algebra} and will often abbreviate the structure by simply writing $(L^{1}(G),*)$.
    \item We let $\mathcal{L}_{\cnv} = \{+,0, \wedge, \vee, \lVert \ \ \rVert,*\}$ be the language of convolution algebras. 
    \item If $X \subseteq G$, we write $\mathbf{1}_{X}$ to denote the characteristic function from $G$ to $\{0,1\}$. If $a \in G$, we write $\mathbf{1}_{\{a\}}$ simply as $\mathbf{1}_{a}$. 
    \item If $G$ is a discrete group, we let $c_{00}(G)$ denote the collection of functions from $G$ to $\mathbb{R}$ with finite support. We will often use the fact that this space is dense in $\ell^{1}(G)$. 
    \item If $G$ is a topological group, we let $C_{c}(G)$ denote the collection of continuous function from $G$ to $\mathbb{R}$ with compact support.
    \item Let $\mathcal{L}$ be a continuous signature and $M,N$ be $\mathcal{L}$-structures. A \textbf{metric embedding} is an isometry $f\colon M \to N$ which is an embedding, i.e., it is an isometry which preserves functions and relation symbols. 
\end{enumerate}

\subsection{Convolution on locally compact groups}  Let $G$ be a locally compact group (possibly discrete), let $\mathcal{B}$ be the Borel $\sigma$-algebra and let $\mu$ be a left Haar measure associated to $G$. We will study $L^{1}(G):=L^{1}(G,\mathcal{B},\mu)$ as a Banach lattice and we will enrich $L^1(G)$ with a product given by convolution,
\begin{equation*}(f*g)(x)=\int_Gf(y)g(y^{-1}x)d\mu(y).
\end{equation*}
To deal with the expansion $(L^{1}(G),*)$ model theoretically, we need to ensure that the new operations are uniformly continuous and that the convolution of two functions from $L^{1}(G)$ is again an object in $L^{1}(G)$.

\begin{fact}[Young's inequality] Let $G$ be a locally compact group. Assume that $f,g\in L^{1}(G)$, then $f*g\in L^{1}(G)$ and furthermore 
    $\lVert f*g\rVert_1\leq \lVert f\rVert_1\lVert g\rVert_1$.
\end{fact}

From the previous result it follows easily that the convolution is uniformly continuous in the unit ball:

\begin{corollary} Let $G$ be a locally compact group. Then the operation given by convolution is uniformly continuous on the unit ball of $L^{1}(G)$ and its modulus of uniform continuity in $\Delta_*(\epsilon)=\epsilon/2$.
\end{corollary}

\begin{proof}
Let $\epsilon>0$ and let $f_1,f_2,g_1,g_2\in L^{1}(G)$ be such that $\lVert f_1\rVert_1$, $\lVert f_2\rVert_1$, $\lVert g_1\rVert_1$, $\lVert g_2\rVert_1 \leq 1$ and $\lVert f_1-f_2\rVert_1 <\epsilon/2$, $\lVert g_1-g_2\rVert_1 <\epsilon/2$. Then 
$\lVert f_1*g_1-f_2*g_2\rVert _1\leq \lVert f_1*g_1-f_2*g_1\rVert _1+\lVert f_2*g_1-f_2*g_2\rVert_1\leq \lVert f_1-f_2\rVert_1\lVert g_1\rVert_1+\lVert f_2\rVert_1\lVert g_1-g_2\rVert_1<\epsilon/2+\epsilon/2=\epsilon$.
\end{proof}

To deal with $(L^{1}(G);+,0,\wedge, \vee, \lVert \ \ \rVert,*)$ as a metric structure, we work  in the unit ball of the space $L^{1}(G)$, and we may formally deal with the operations $+$, $\wedge$, $\vee$ as binary maps that take elements in the unit ball to elements in the ball of radius two. Similarly, the norm $\lVert \ \ \rVert$ induces a metric given by $d(f,g)=\lVert f-g \rVert_1$ which has values in $[0,2]$. Equivalently, the reader may prefer to divide by $2$ when taking the operations $+$, $\wedge$, $\vee$ to ensure that the resulting elements remain within the unit ball and divide by $2$ when defining the distance inside the ball to ensure the predicate has values in $[0,1]$.

In the final section, we establish an embedding theorem, showing that $(L^{1}(A),*)$ admits a linear metric embedding into the product $\prod_{D} (\ell^{1}(H_i),*)$, where $A$ is a connected abelian Lie group, $(H_i)_{i < \omega}$ is a sequence of finite abelian groups, and $D$ is a non-principal ultrafilter on $\mathbb{N}$. Since we are dealing with linear metric embeddings, the unit ball of $L^{1}(A)$ embeds into the unit ball of $\prod_{D} (\ell^{1}(H_i),*)$, so again, even though we are dealing with Banach spaces, we may formally work within the unit balls and extend the maps linearly to the whole spaces.

\subsection{(Discretely) amenable groups} 

Finally, we recall a classical fact about discretely amenable groups. There are many equivalent definitions for these groups, but for our purposes, the following two suffice. A proof of the following equivalence can be found in \cite{Namioka} (see Theorem 2.2 and Corollary 4.3; for groups, it is straightforward to prove that \emph{left amenable} is equivalent to \emph{right amenable} by inverting the family of sets in F\o lner's condition). 

\begin{definition}\label{def:FC}  Let $G$ be a discrete group. We say that $G$ is amenable if one of the following equivalent conditions holds:
\begin{enumerate} 
\item There exists a net $\{h_{\gamma}\}_{\gamma \in \Gamma}$ of elements in $\ell^{1}(G)$ such that for each $\gamma \in \Gamma$, $|\supp(h_\gamma)|$ is finite, $\lVert h_{\gamma} \rVert = 1$, and for every $a \in G$,
\begin{equation*} 
\lim_{\gamma \in \Gamma} \lVert h_{\gamma}*\mathbf{1}_a - h_{\gamma} \rVert = 0. 
\end{equation*} 
\item (F\o lner's criterion) There exists a net $(X_i)_{i \in I}$ of finite subsets of $G$ such that for every $g \in G$,
\begin{equation*}
    \lim_{i \in I} \frac{|(g\cdot X_i) \triangle X_i|}{|X_i|} = 0. 
\end{equation*}
If $G$ is countable and satisfies F\o lner's criterion, then the net may be replaced by a sequence.
\end{enumerate} 
\end{definition} 

\section{Banach lattices, definability of atoms and transfer theorems}\label{sec:atoms}

In this section we first prove some basic facts regarding Banach lattices of the form 
$(L^1(X, \mathcal{B}, \mu); +, 0, \wedge, \vee,\lVert \ \ \rVert)$. We prove the definability of the collection of characteristic functions of the atoms of the Banach lattice. When $X$ is equal to a discrete group $G$ and $\mu$ is the counting measure, we establish a transfer principle between the convolution algebra associated to $G$ and $G$ itself. In particular, we show that if $(\ell^{1}(G),*) \preceq (\ell^{1}(H),*)$ or $(\ell^{1}(G),*) \equiv (\ell^{1}(H),*)$, then $(G,\cdot) \preceq (H,\cdot)$ or $(G,\cdot) \equiv (H,\cdot)$, respectively. Assuming $(G,\cdot)$ and $(H,\cdot)$ are $\omega$-saturated, we prove a converse to this result. We conclude the section by giving examples of groups $G$ and $H$ such that $G \preceq H$ but $(\ell^{1}(G),*) \not \equiv (\ell^{1}(H),*)$; in particular, transfer fails whenever $G$ is amenable but not uniformly amenable.

\subsection{Defining the atoms} Recall that a measure space $(X,\mathcal{B},\mu)$ is \emph{decomposable} if there is a partition $\{X_i
: i \in I\}\subseteq \mathcal{B}$ of $X$ into measurable sets such that $\mu(X_i) < \infty$ for all $i \in I$
and such that for any subset $A$ of $X$, $A \in \mathcal{B}$ iff $A \cap X_i \in \mathcal{B}$ for all $i \in I$ and $\mu(A) =\sum_{i\in I} \mu(A \cap X_i)$. The reader can check \cite[section 2]{BYBHLp} for more details on decomposable spaces and abstract $L_p$ spaces. We may assume that the space $(X,\mathcal{B},\mu)$ is a disjoint union of two spaces, one of them \emph{atomless} and another consisting only of \emph{atoms}. Furthermore, we may take the atomic part in such a way that each atom has measure one, i.e., it has the counting measure over a discrete set as every atomic $L_p$ space can be represented in this way, see for example \cite[Proposition 1.6.4]{CeMe}.

In this subsection, we prove that the atomic part of $(L^{1}(X,\mathcal{B},\mu),+,0,\wedge,\vee,\lVert  \ \lVert )$ is definable in the sense of continuous logic. Most of the results from this subsection are folklore and were known to hold at least by Ben Yaacov and Henson, but they have not appeared in print. For completeness of the presentation we include them here. 

\begin{notation}
For $f\in L^{1}(X,\mathcal{B},\mu)$, we write $f^+=f\vee 0$ and $f^-=(-f)\vee 0$. We call $f^+$ the \textbf{positive part} of $f$ and $f^-$ the \textbf{negative part} of $f$.
\end{notation}

Note that for $f\in L^{1}(X,\mathcal{B},\mu)$, we have $f=f^+-f^-$.

\begin{definition}
We say a function $f\in L^{1}(X,\mathcal{B},\mu)$ is an \textbf{atom} if $f=\mathbf{1}_{a}$ for some $a\in X$ and $\lVert f \rVert = 1$.
\end{definition}

Note that these \emph{atoms} are precisely the characteristic functions of singleton sets that are atoms of the underlying measure space.

\begin{lemma}
    An element $f\in (L^{1}(X,\mathcal{B},\mu),+,0,\wedge,\vee,\lVert  \ \ \rVert )$ is an atom if and only if $f\geq 0$, $\lVert f\rVert =1$ and for all $n\in \mathbb{N}^{>0}$, and all functions $g$ satisfying $0\leq g\leq f$ with $\lVert g\rVert \geq 1/n$ we have $ng\geq f$.  
\end{lemma}

\begin{proof}
Assume first that $f \in (L^{1}(X,\mathcal{B},\mu),+,0,\wedge,\vee,\lVert  \ \ \rVert )$ is an atom. Then $f=\mathbf{1}_{a}$ with $\lVert f\lVert =1$. Then $f\geq 0$. If $g$ is a function with $0\leq g\leq f$ and $\lVert g\lVert \geq 1/n$, then $g=r \mathbf{1}_{a}$ with $r\geq 1/n$, so $ng\geq \mathbf{1}_{a}=f$.

Now assume that $f \in (L^{1}(X,\mathcal{B},\mu),+,0,\wedge,\vee,\lVert  \ \ \rVert )$ is not an atom. Then $f=r_1 h_1+ r_2h_2$ with $h_1\perp h_2$ (i.e. they have disjoint support), $\lVert h_1\lVert =\lVert h_2\lVert =1$, $r_1,r_2>1/n$ for some $n$ and $r_1+r_2=1$. Let $g=r_1 h_1$, then $0\leq g\leq f$ and $\lVert g\lVert \geq 1/n$. But $ng \perp h_2$ and thus $ng\ngeq f$.
\end{proof}

The previous argument shows that $f$ being an atom is part of the information contained in $\tp(f/\emptyset)$. We will now prove a stronger result, namely the definability of the set of atoms. We will use the following criterion for the definability of a set (see \cite[Proposition 9.19]{BBHU2008}).

\begin{fact}\label{test-definibility} Let $M$ be a metric $\mathcal{L}$-structure. For a closed set $D\subseteq M$, the following are equivalent:\\
(1) $D$ is definable in $M$ over $\emptyset$. \\
(2) There is a sequence $(\varphi_n(x) | n \geq  1)$ of $\mathcal{L}$-formulas and a sequence
$(\delta_n | n \geq 1)$ of positive real numbers such that for all
$n \geq 1$ and $x\in M$,
\begin{equation*} x \in D \Longrightarrow \varphi_n(x) = 0,
\end{equation*}
and,
\begin{equation*} \varphi_n(x) \leq \delta_n \Longrightarrow dist(x,D)<1/n. 
\end{equation*}
\end{fact}

Recall that if a set is definable over $\emptyset$, then quantifying over the set will give definable predicates
that can be treated as formulas in the language.

\begin{lemma}\label{def-positive functions} For a Banach lattice $L^{1}(X,\mathcal{B},\mu)$, the set of positive functions of norm one is a definable set and thus a quantifiable set.   
\end{lemma}

\begin{proof}
Let $\theta(x)=\max\{|\lVert x\lVert -1|,|\lVert x^+\lVert -\lVert x\lVert |\}$. If $f$ is a positive function of norm one, then $\theta(f)=0$. 

Now let $0<\epsilon<1/2$ and let $\delta(\epsilon)\leq \epsilon$ be such that if $1-\delta<\lVert f\lVert <1+\delta$ then $d(f,f/\lVert f\lVert)<\epsilon$. If $\theta(f)<\delta/2$ then 
  $1-\delta/2<\lVert f\lVert <1+\delta/2$
  and $\lVert f\lVert -\lVert f^+\lVert <\delta/2$
  so $1-\delta<\lVert f^+\lVert <1+\delta$ and $d(f,f^+/\lVert f^+\lVert )\leq d(f,f^+)+d(f^+,f^+/\lVert f^+\lVert )<\delta+\epsilon\leq 2\epsilon$.
\end{proof}

\begin{proposition}\label{def-atoms} For a Banach lattice $L^{1}(X,\mathcal{B},\mu)$, the set of functions $D=\{f: f$ is an atom$\}$ is definable and thus a quantifiable set.
\end{proposition}

\begin{proof} We will apply Fact \ref{test-definibility}
with $D=\{f: f$ is an atom$\}$. For each positive integer $n>0$ consider the formulas:
$$\psi_n(x) = \sup_{\substack{\lVert y\lVert =1 \\ y\geq 0}} 
\min\{\lVert x\wedge y\lVert \dotdiv 1/n,\lVert (x-n(x\wedge y))^+\lVert \} $$

$$\varphi_n(x) = \max\{|\lVert x^+\lVert -1|, \psi_n(x^+),| \lVert x\lVert -\lVert x^+\lVert |\} $$
and the sequence $(\varphi_n|n\geq 1)$. 

\begin{claim-star} If $x=\mathbf{1}_{a}$ is an atom, then for all $n>0 $ we have $\varphi_n(x)=0$. 
\end{claim-star}
\begin{claimproof} First observe that $x^+=x$ , and thus $\lVert  x^+\lVert =\lVert x\lVert =1$, and so the first and third terms inside of $\varphi_{n}$ are zero. For any $y \geq 0$ with $\lVert y \rVert = 1$, if $\lVert x\wedge y\lVert \leq 1/n$ then $\min\{\lVert x\wedge y\lVert \dotdiv 1/n,\lVert (x-n(x\wedge y))^+\lVert \}=0$. Otherwise $\lVert x \wedge y \rVert > 1/n$ implies $x \wedge y > (1/n) \mathbf{1}_{a}$, and so $n(x \wedge y) \geq x$ and $(x - n(x \wedge y))^{+} = 0$. Thus $\psi_{n}(x) = 0$ and so $\varphi_{n}(x) = 0$. 
\end{claimproof}

\begin{claim-star} If $\varphi_n(x)<\frac{1}{n^2}$ and $n\geq 2$ then $dist(x,D)\leq 2/n+4/n^2< 6/n$.
\end{claim-star}

\begin{claimproof} First note that $1+1/n^2\geq \lVert x^+\lVert \geq 1-1/n^2$ and that $d(x,x^+)=\lVert x\lVert-\lVert x^+\lVert<\frac{1}{n^2}$. Write $x^+=h_1+h_2$ with 
$h_1\perp h_2$ and suppose, in order to obtain a contradiction, that we have both $\lVert h_1\lVert ,\lVert h_2\lVert >1/n+1/n^2$. 
Since $x^+$ is positive, we must have $h_1\geq 0$ and $h_2\geq 0$. Since $\lVert x^+\lVert \leq 1+1/n^2$ and $h_1,h_2$ are disjoint then $\lVert h_1\lVert ,\lVert h_2\lVert \leq 1$. Let $g_1=h_1/\lVert h_1\lVert $ and $g_2=h_2/\lVert h_2\lVert $, both are positive functions of norm one. Then $g_1\geq h_1$, $g_2\geq h_2$, $g_1\wedge x^+=h_1$, $g_2\wedge x^+=h_2$ and $\lVert x^+\wedge g_i\lVert > 1/n+1/n^2$ for $i=1,2$. Similarly $(x^+-n(x^+\wedge g_1))^+=(x^+-nh_1)^+=h_2$. 

Now choose $y=g_1$, we have $y\geq 0$, $\lVert y\lVert =1$,
$\lVert x^+\wedge y\lVert - 1/n>1/n^2$ and
$\lVert (x-n(x^+\wedge y))^+\lVert =\lVert h_2\lVert >1/n+1/n^2$, which contradicts $\psi_n(x^+)\leq\varphi_n(x)\leq\frac{1}{n^2}$. 
\end{claimproof}

Thus, for any decomposition $x^+=h_1+h_2$ with 
$h_1\perp h_2$ we must have $\lVert h_i\lVert \leq 1/n+1/n^2$
for some $i$. This implies that $x^+=r\mathbf{1}_{a}+h$ with $\lVert h\lVert \leq 1/n+1/n^2$ for some $r>0$ and $h\perp \mathbf{1}_{a}$. Since $\lVert x^+\lVert \geq 1-1/n^2$ we must have that $r>1-1/n-2/n^2$. Since $\lVert x^+\lVert \leq 1+1/n^2$ we must also have $r\leq \lVert r\mathbf{1}_{a}\lVert \leq \lVert x^+ \lVert \leq 1+1/n^2$. Thus $|r-1|<\frac{1}{n}+\frac{1}{n^2}$. Then, 
\begin{align*}
dist(x,D)&\leq dist(x,x^+)+dist(x^+,D)\leq 1/n^2+dist(x^+,\mathbf{1}_{a})\\
&\leq 1/n^2+|1-r|+\lVert h\lVert \leq 2/n+4/n^2. \qedhere
\end{align*}
\end{proof}

\begin{remark} If $(X,\mathcal{B},\mu)$ is atomless, then for $n \geq 4$ and $f \in L^{1}(X,\mathcal{B},\mu)$, we have $\varphi_{n}(f) \geq 1/8$ and thus the formula shows that the distance to the collection of atoms is uniformly bounded from $0$, i.e., $D=\emptyset$.

We may write $f=f^+-f^-$. Let $A=\{x: f(x)>0\}$ and $C=\{x: f(x)<0\}$. If $\lVert f\cdot \mathbf{1}_C\rVert\geq 1/8$ then $\lVert f\rVert -\lVert f^+\rVert\geq 1/8$ and $\varphi_n(f)\geq 1/8$ for all $n\geq 1$. If $\lVert f\cdot \mathbf{1}_A\rVert\leq 7/8$ or $\lVert f\cdot\mathbf{1}_A\rVert\geq 9/8$ then $|\lVert f^+\lVert-1|\geq 1/8$ and $\varphi_n(f)\geq 1/8$ for all $n\geq 1$.
    Assume then that $7/8\leq \lVert f\cdot \mathbf{1}_A\rVert\leq 9/8$ and $\lVert f\cdot \mathbf{1}_C\rVert\leq 1/8$. Let $B\subseteq A$ be such that $\lVert f\cdot \mathbf{1}_B\rVert=\frac{1}{2}\lVert f\cdot \mathbf{1}_A\rVert$, then we also have $\lVert f\cdot \mathbf{1}_{A\setminus B}\rVert=\frac{1}{2}\lVert f\cdot \mathbf{1}_A\rVert$ and notice that $\lVert f\cdot \mathbf{1}_B\rVert<1$. Take $g=\frac{f\cdot\mathbf{1}_{B}}{\lVert f\cdot \mathbf{1}_B\rVert}$, then $g\geq 0$, $\lVert g\lVert=1$ and 
    $f\wedge g=f\cdot\mathbf{1}_{B}$.
    For $n\geq 1$ we have $(f-n(f\wedge g))^+=f\cdot\mathbf{1}_{A\setminus B}$ and thus $\lVert (f-n(f\wedge g))^+\lVert= \lVert f\cdot \mathbf{1}_{A\setminus B}\rVert\geq 7/16$. On the other hand, for $n\geq 4$ we get $\lVert f\wedge g\lVert-1/n\geq \lVert f\cdot \mathbf{1}_B\rVert-1/n\geq 7/16-1/4=3/16$ and thus $\psi_n(f)\geq \min(7/16,3/16)\geq 1/8$ and we get $\varphi_n(f)\geq 1/8$.
\end{remark}

Finally, we take the opportunity to discuss the structure of non-standard models of convolution algebras. To understand these models, we first need the following result.

\begin{lemma}\label{unifdistanceatoms}
 Let $G$ be a group and let $f,g\in \ell_1(G)$. Let $\mathcal{A}$ be the set of atoms of $\ell_1(G)$. If $f,g\geq 0$ and $\lVert f\rVert=\lVert g\rVert=1$, then $d(f*g,\mathcal{A})\geq d(f,\mathcal{A})$. 
\end{lemma}

\begin{proof}
Assume that $f,g$ have finite support. Say $f=\sum_{i=1}^m k_i\mathbf{1}_{g_i}$ and $g=\sum_{j=1}^nl_j\mathbf{1}_{h_j}$. Since $f,g\geq 0$ and $\lVert f\rVert=\lVert g\rVert=1$, we have that $k_i,l_j\geq 0$ for each $i,j$ and $\sum_{i=1}^mk_i=\sum_{j=1}^nl_j=1$. Note that
$$d(f,\mathcal{A})=\min_{1\leq i\leq m}\{1-k_i+k_1+k_2+\cdots+\widehat{k_i}+\cdots+k_m\}=\min_{1\leq i\leq m}\{2-2k_i\}.$$
Also, $f*g=\sum_{1\leq i\leq m,1\leq j\leq n}k_il_j\mathbf{1}_{g_ih_j}$. We will consider two cases:

\textbf{Case 1.} Assume all elements in the family $(g_ih_j:1\leq i\leq m,1\leq j\leq n)$ are distinct. Then $$d(f*g,\mathcal{A})=\min_{i,j}\{1-k_il_j+k_1l_1+\cdots+\widehat{k_il_j}+k_ml_n\}=\min_{i,j}\{2-2k_il_j\},$$ and thus $$d(f*g,\mathcal{A})=\min_{i,j}\{2-2k_il_j\}\geq\min_{i}\{2-2k_i\}=d(f,\mathcal{A}).$$
\textbf{Case 2.} Assume some the elements $g_ih_j$ are equal, say $g_{i_1}h_{j_1}=\cdots=g_{i_p}h_{j_p}=a$. Then, it is easy to verify that 
$$d(f*g,\chi_a)=2-2k_{i_1}l_{j_1}-\cdots-2k_{i_p}l_{j_p}.$$
Note that $k_{i_1}l_{j_1}+\cdots+k_{i_p}l_{j_p}\leq\max_i k_i(l_{j_1}+\cdots+l_{j_p})\leq \max_i k_i$. Thus $$d(f*g,\mathcal{A})=\min_a d(f*g,\mathbf{1}_a)\geq \min_a\{2-2\max_i k_i\}= d(f,\mathcal{A}).$$

Since the family of functions $f,g$ with finite support are dense and the distance function to the collection of atoms is continuous, the result follows by taking limits.
\end{proof}

\begin{observation} Consider a discrete group $G$ and its associated convolution algebra $(\ell^1(G),+,0,\wedge,\vee,*,\lVert \cdot\rVert)$ and let $\mathcal{A}_{\ell^1(G)}$ be the set of atoms of $(\ell^1(G),*)$. Let $(B,+,0,\wedge,\vee,*,\lVert \cdot\rVert)\equiv (\ell^1(G),+,0,\wedge,\vee,*,\lVert \cdot \rVert)$, and let $\mathcal{A}_B$ be the set of atoms of the structure $B$. Since the convolution defines a group operation on $\mathcal{A}_{\ell^1(G)}$ (which is isomorphic to the group $G$), the multiplication in the algebra $(B,*)$ defines a group operation on $\mathcal{A}_B$ and we write $(G_B,\cdot)$ for the group associated to $(\mathcal{A}_B,*)$ (see Subsection \ref{subsec:defininggroup} for details on the definability of the discrete group $(G_B,\cdot)$). We may write
       $$(B,*)= (\ell^1(G_B) \oplus L^1(X_B),*)$$
       and we call $L^1(X_B)$ the \textbf{atomless part} of $B$ and $\ell^1(G_B)$ the \textbf{atomic part} of $B$. Both of these spaces are Banach lattices and we write the direct sum as an $\ell^1$ sum of the two spaces, in the sense that if $f\in \ell^1(G_B)$, $g\in L^1(X_B)$, then $\lVert f+g\rVert=\lVert f\rVert+\lVert g\rVert$ (see for example \cite[section 2]{He76} for more details). 

        Since the convolution of two atoms is again an atom, the atomic part is closed under $*$.
  By Lemma \ref{unifdistanceatoms}, the convolution of two elements in $L^1(X_B)$ belongs again to $L^1(X_B)$. Finally, it also follows from Lemma \ref{unifdistanceatoms}, that the convolution of an element in $L^1(X_B)$ with an element in $\ell^1(G_B)$ belongs again to $L^1(X_B)$. Thus, $L^1(X_B)$ is an ideal in $B$.

Consider the special case when $G$ is infinite discrete, $\{g_n:n\geq 1\}$ are different elements in $G$, $\mathcal{U}$ is a non-principal ultrafilter over $\mathbb{N}$ and define $B=\Pi_\mathcal{U}(\ell^1(G),*)$ and the functions $f_n= \frac{1}{n}\sum_{i=1}^n \mathbf{1}_{g_i}$. Let $[f]_\mathcal{U}=(f_n)_\mathcal{U}\in B$ which has norm one. Note  that $dist(f_n,\mathcal{A}_{\ell_1(G)})=\frac{n-1}{n}$ and thus $[f]_\mathcal{U}\in L^1(X_B)$, so in a saturated model $L^1(X_B)\neq \emptyset$ and in fact $L^1(X_B)$ is infinite dimensional. Note that the ideal $L^1(X_B)$ is not definable as in it omitted in $(\ell^1(G),*)$ and infinite dimensional in $\Pi_\mathcal{U}(\ell^1(G),*)$.
\end{observation}

\begin{question}The Banach algebra $B=\Pi_\mathcal{U}(\ell^1(G),*)$ contains an atomic part and an atomless part and thus this new structure is not isomorphic to $(L^{1}(H),*)$ for a group $H$. Is there an analytic description of the class of Banach algebras which are elementarily equivalent to $(\ell^1(G),*)$? 
\end{question}

\subsection{Convolution equivalence implies group equivalence}\label{subsec:defininggroup}

In this subsection, we prove that if $G$ is a subgroup of $H$ and $(\ell^{1}(G),*) \preceq (\ell^{1}(H),*)$, then $(G,\cdot) \preceq (H,\cdot)$ as pure groups. Likewise, if $(\ell^{1}(G),*) \equiv (\ell^{1}(H),*)$ then $(G,\cdot) \equiv (H,\cdot)$ again in the pure group language. To prove these statements, we will define a map from the language of pure groups to the language of convolution algebras and induct on the complexity of formulas. We first define a map from discrete $\mathcal{L}_{\grp}$-formulas to continuous $\mathcal{L}_{\cnv}$-formulas. For the purpose of this section only, we will use $x,y,z$ to denote discrete variables (associated to the language $\mathcal{L}_{\grp}$) and $\mathbf{x},\mathbf{y},\mathbf{z}$ to denote continuous variables (associated to the language $\mathcal{L}_{\cnv}$). Throughout, we let $D(\mathbf{x})$ denote the zeroset for the atoms in a convolution algebra. By Proposition \ref{def-atoms}, this is a definable set.

\begin{definition} We define a map from discrete $\mathcal{L}_{\grp}$-formulas in variables $\{x_i: i < \omega\}$ to $\mathcal{L}_{\cnv}$-formulas in variables $\{\mathbf{x}_i: i < \omega\}$ inductively. First, we define a map $\gamma$ from $\mathcal{L}_{\grp}$-terms to $\mathcal{L}_{\cnv}$-terms. We then define an induced map $\Gamma$ from $\mathcal{L}_{\grp}$-formulas to $\mathcal{L}_{\cnv}$-formulas. 
\begin{enumerate} 
\item For any discrete variable $x_i$, we let $\gamma(x_i) = \mathbf{x}_i$. 
\item If $t_1(\bar{x})$ and $t_2(\bar{y})$ are $\mathcal{L}_{\grp}$-terms in which $\gamma$ has been defined, then we define $\gamma[t_1(\bar{x}) \cdot t_2(\bar{y})]= \gamma(t_1(\bar{x})) * \gamma(t_2(\bar{y}))$.  
\item If $t_1(\bar{x})$ and $t_2(\bar{y})$ are $\mathcal{L}_{\grp}$-terms, then $\Gamma[t_1(\bar{x}) = t_2(\bar{y})] = d(\gamma(t_1(\bar{x})),\gamma(t_2(\bar{y})))$. 
\item If $\psi(\bar{x})$ and $\psi(\bar{y})$ are $\mathcal{L}_{\grp}$-formulas in which $\Gamma$ has been defined, then we define $\Gamma(\phi \wedge \psi) = \max\{\Gamma(\phi), \Gamma(\psi)\}$.
\item  If $\psi(\bar{x})$ is an $\mathcal{L}_{\grp}$-formula in which $\Gamma$ has been defined, then we define $\Gamma(\neg \psi) = 1 - \Gamma(\psi)$. 
\item If $\psi(\bar{x})$ is an $\mathcal{L}_{\grp}$-formula in which $\Gamma$ has been defined, then we define $\Gamma(\exists y\psi(\bar{x},y)) = \inf_{y \in D} \Gamma(\psi(\bar{x},y)).$
\end{enumerate} 
If $\varphi(x_1,...,x_n)$ is an $\mathcal{L}_{\grp}$-formula, then we denote the associated $\mathcal{L}_{\cnv}$-formula $\Gamma(\varphi(x_1,...,x_n))$ by  $\hat{\varphi}(\mathbf{x}_1,...,\mathbf{x}_n)$. 
\end{definition} 

\begin{lemma}\label{lemma:iso_atomic} Fix a discrete group $G$. The map $g \mapsto \mathbf{1}_{g}$ from $G$ to $D(\ell^{1}(G))$ is an isomorphism of pure groups. 
\end{lemma} 

\begin{proof} This follows directly from the definition of convolution. Notice that 
\begin{equation*} 
(\mathbf{1}_{g_1} * \mathbf{1}_{g_2})(x) = \int_{x \in G} \mathbf{1}_{g_1}(y) \mathbf{1}_{g_2}(y^{-1} x) dm = \mathbf{1}_{g_2}(g_1^{-1} \cdot x) = \mathbf{1}_{g_1 \cdot g_2}(x). \qedhere
\end{equation*} 
\end{proof} 

\begin{lemma}\label{lemma:01} Fix a discrete group $G$. If $g_1,...,g_n \in G$ then for any $\mathcal{L}_{\grp}$-formula $\varphi(x_1,...,x_n)$, 
$ \hat{\varphi}(\mathbf{1}_{g_1},...,\mathbf{1}_{g_n})^{\ell^1(G)}\in \{0,1\}$.
\end{lemma} 

\begin{proof} We proceed by induction on formula complexity. The atomic case follows immediately by the way we defined the distance between functions and from Lemma~\ref{lemma:iso_atomic}, while the propositional connectives are routine. It remains to consider the existential quantifier. Suppose $\varphi(\bar{x}) = \exists y \, \psi(\bar{x}, y)$ and, as our induction hypothesis, assume that for every tuple $(\bar{g}, h) \in G^{|\bar{x}|} \times G$, the truth value of $\hat{\psi}(\mathbf{1}_{g_1},...,\mathbf{1}_{g_n}, \mathbf{1}_{h})$ relative to $(\ell^{1}(G), *)$ is either $0$ or $1$. Then
\begin{align*}
\hat{\varphi}^{\ell^{1}(G)}(\mathbf{1}_{g_1}, \ldots, \mathbf{1}_{g_n}) 
&= \inf_{\mathbf{y} \in D} \hat{\psi}^{\ell^{1}(G)}(\mathbf{1}_{g_1}, \ldots, \mathbf{1}_{g_n}, \mathbf{y}) \\
&= 
\begin{cases}
1 & \text{if for every } a \in D(\ell^{1}(G)),\ \hat\psi(\mathbf{1}_{g_1}, \ldots, \mathbf{1}_{g_n}, a) = 1, \\[6pt]
0 & \text{if there exists } a \in D(\ell^{1}(G)),\ \hat\psi(\mathbf{1}_{g_1}, \ldots, \mathbf{1}_{g_n}, a) = 0. \qedhere
\end{cases}
\end{align*} 
\end{proof} 

\begin{corollary}\label{inf=min} Let $G$ be a discrete group. If $g_1,...,g_n \in G$ then for every $\mathcal{L}_{\grp}$-formula $\psi(x_1,...,x_n,y)$,  $(\ell^1(G),*) \models \inf_{\mathbf{y}\in D}\hat\psi(\mathbf{1}_{g_1},...,\mathbf{1}_{g_n},\mathbf{y})$ if and only if  $(\ell^{1}(G),*) \models \hat{\psi}(\mathbf{1}_{g_1},...,\mathbf{1}_{g_n},\mathbf{1}_g) = 0$ for some $g\in G$. 
\end{corollary} 
\begin{proof}
Directly from Lemma \ref{lemma:01}.
\end{proof}

\begin{lemma}\label{lemma:good} Fix a discrete group $G$. If $g_1,...,g_n \in G$ then for any $\mathcal{L}_{\grp}$-formula $\varphi(x_1,...,x_n)$, 
\begin{equation*} (G,\cdot) \models \varphi(g_1,...,g_n) \Longleftrightarrow (\ell^{1}(G),*) \models \hat{\varphi}(\mathbf{1}_{g_1},...,\mathbf{1}_{g_n}) = 0. \
\end{equation*} 
\end{lemma} 

\begin{proof} Again, we proceed by induction on formula complexity. The atomic case follows immediately from the way we defined the distance between functions and Lemma~\ref{lemma:iso_atomic}, propositional connectives are routine. We again check the existential quantification case.  Suppose $\varphi(\bar{x}) = \exists y \, \psi(\bar{x}, y)$ and, as our induction hypothesis, assume that for every tuple $(\bar{g}, h) \in G^{|\bar{x}|} \times G$, 
\begin{equation*} 
G \models \psi(g_1,...,g_n,h) \Longleftrightarrow (\ell^{1}(G),*) \models \hat{\psi}(\mathbf{1}_{g_1},...,\mathbf{1}_{g_n},\mathbf{1}_h) =0.
\end{equation*} 
Then 
\begin{align*}
G \models \varphi(g_1,...,g_n) &\Longleftrightarrow G \models \exists y \psi(g_1,...,g_n,y) \\ 
&\Longleftrightarrow  G \models \psi(g_1,...,g_n,h) \text{ for some $h\in G$}\\
&\Longleftrightarrow  (\ell^{1}(G),*) \models \hat{\psi}(\mathbf{1}_{g_1},...,\mathbf{1}_{g_n},\mathbf{1}_{h}) = 0 \text{ for some $h\in G$}\\
&\Longleftrightarrow (\ell^{1}(G),*) \models \inf_{\mathbf{y} \in D} \hat{\psi}(\mathbf{1}_{g_1},...,\mathbf{1}_{g_n},\mathbf{y}) = 0 \\
&\Longleftrightarrow (\ell^{1}(G),*) \models  \hat{\varphi}(\mathbf{1}_{g_1},...,\mathbf{1}_{g_n}) = 0. 
\end{align*} 
The fourth $(\Longleftrightarrow)$ follows from Corollary \ref{inf=min}.
\end{proof} 

\begin{theorem}  Fix discrete groups $G,H$ such that $G \subseteq H$. If $(\ell^{1}(G),*) \preceq (\ell^{1}(H),*)$, then $(G,\cdot) \preceq (H,\cdot)$. 
\end{theorem} 

\begin{proof}  We prove the statement via Tarski-Vaught. Fix a $\mathcal{L}_{\grp}$-formula $\varphi(x_1,...,x_n,y)$ and elements $g_1,...,g_n \in G$. Then, 
\begin{align*} 
H \models \exists y\varphi(g_1,...,g_n,y) 
&\Longrightarrow H \models \varphi(g_1,...,g_n,h) \text{ for some $h\in H$}\\
&\Longrightarrow (\ell^{1}(H),*) \models \hat{\varphi}(\mathbf{1}_{g_1},...,\mathbf{1}_{g_n},\mathbf{1}_{h}) = 0 \text{ for some $h\in H$}\\ 
&\Longrightarrow (\ell^{1}(H),*) \models \inf_{\mathbf{y}\in D}\hat{\varphi}(\mathbf{1}_{g_1},...,\mathbf{1}_{g_n},\mathbf{y})\\
&\Longrightarrow (\ell^{1}(G),*) \models \inf_{\mathbf{y}\in D}\hat{\varphi}(\mathbf{1}_{g_1},...,\mathbf{1}_{g_n},\mathbf{y})\\
&\Longrightarrow (\ell^{1}(G),*) \models \hat{\varphi}(\mathbf{1}_{g_1},...,\mathbf{1}_{g_n},\mathbf{1}_{g_{\star}}) = 0 \text{ for some $g_{\star}\in G$}\\ 
&\Longrightarrow G \models \varphi(g_1,...,g_n,g_{\star}) \text{ for some $g_{\star}\in G$}\\
&\Longrightarrow G \models \exists y\varphi(g_1,...,g_n,y). 
\end{align*} 
The second and sixth $(\Rightarrow)$ follow from Lemma \ref{lemma:good}. 
The third and fifth $(\Rightarrow)$ follows from Corollary \ref{inf=min}.
\end{proof}

\begin{corollary}\label{interpretgroup} Fix discrete groups $G,H$. If $(\ell^{1}(G),*) \equiv (\ell^{1}(H),*)$, then $(G,\cdot) \equiv (H,\cdot)$. 
\end{corollary} 

\begin{proof} Fix an $\mathcal{L}_{\grp}$-sentence $\varphi$. Then by Lemma \ref{lemma:good}, 
\begin{align*} 
G \models \varphi &\Longleftrightarrow (\ell^{1}(G),*) \models \hat{\varphi} = 0 \Longleftrightarrow (\ell^{1}(H),*) \models \hat{\varphi} = 0 \Longleftrightarrow H \models \varphi.  \qedhere
\end{align*} 
\end{proof} 

\subsection{Group equivalence implies convolution equivalence modulo some saturation.} We first prove that if $(G,\cdot)$ and $(H,\cdot)$ are $\omega$-saturated models of $T$ such that $(G,\cdot) \preceq (H,\cdot)$, then $(\ell^{1}(G),*) \preceq (\ell^{1}(H),*)$. As a corollary, if $G$ and $H$ are $\omega$-saturated models of $T$, then $(\ell^{1}(G),*)  \equiv (\ell^{1}(H),*)$. The following variant of the Tarski-Vaught theorem for continuous logic will be quite useful (\cite[Proposition 4.5]{BBHU2008}). 

\begin{fact}[Tarski-Vaught]\label{tarski-vaught}
Let $S$ be any set of $\mathcal{L}$-formulas that is dense with respect to logical distance. 
Suppose $M$ and $N$ are $\mathcal{L}$-structures with $M \subseteq N$. 
Then the following statements are equivalent:
\begin{enumerate}
    \item $M \preceq N$;
    \item For every $\mathcal{L}$-formula $\phi(x_1, \ldots, x_n, y)$ in $S$ and all $a_1, \ldots, a_n \in M$,
    \[
    \inf \bigl\{ \phi^N(a_1, \ldots, a_n, b) \mid b \in N \bigr\}
    = 
    \inf \bigl\{ \phi^N(a_1, \ldots, a_n, c) \mid c \in M \bigr\}.
    \]
\end{enumerate}
\end{fact}

The main technical lemma of the section is the following. 

\begin{lemma}\label{lemma:auto} Suppose that $(G,\cdot) \preceq (H,\cdot)$ and for every $a_1,...,a_n \in G$ and $b_1,...,b_m \in H$, there exists an automorphism $\sigma\colon H \to H$ such that for each $i \leq n$, $\sigma(a_i) = a_i$ and for each $j \leq m$, $\sigma(b_j) \in G$. Then $(\ell^{1}(G),*) \preceq (\ell^{1}(H),*)$. 
\end{lemma}

\begin{proof}
 We use Tarski-Vaught to prove $(\ell^{1}(G),*) \preceq (\ell^{1}(H),*)$, so fix a formula $\varphi(\bar{x},y)$. Choose $f_1,...,f_n \in \ell^{1}(G)$. Suppose that $\inf_{y \in \ell^{1}(H)}\varphi^{\ell^{1}(H)}(f_1,...,f_n,y) = r$.  Choose $\epsilon>0$ and let $h\in \ell^1(H)$ be such that $\varphi^{\ell^{1}(H)}(f_1,...,f_n,h) < r+\epsilon$. Since that map $\varphi$ is uniformly continuous, there exists some $\delta(\epsilon)$ such that if $\max_{i \leq n} \{d(t_i,s_i),d(v,w)\} < \delta(\epsilon)$ then $|\varphi^{\ell^{1}(H)}(t_1,...,t_n,v) - \varphi^{\ell^{1}(H)}(s_1,...,s_n,w)| < \epsilon$. Using metric density of functions with finite support, we may find $\tilde{f}_{1},...,\tilde{f}_{n},\tilde{h}$ such that 
\begin{enumerate} 
    \item For each $i \leq n$, $\supp(\tilde{f}_i) \subseteq G$ and $|\supp(\tilde{f}_{i})| < \aleph_0$. 
    \item Also $\supp(\tilde{h}) \subseteq H$ and $|\supp(\tilde{h})| < \aleph_0$.
    \item  $\max_{i \leq n} \{d(f_i,\tilde{f}_i),d(h,\tilde{h})\} < \delta(\epsilon)$ and so $\varphi^{\ell^{1}(H)}(\tilde{f}_{1},...,\tilde{f}_{n},\tilde{h}) < r + 2\epsilon$.
\end{enumerate}
Fix $\tilde{f}_{1},...,\tilde{f}_{n},\tilde{h}$ as above. Consider $A_0 = \bigcup_{i \leq n} \supp(\tilde{f}_i)$. Let $B = \{b_1,...,b_{k}\} = \supp(\tilde{h})$. By our hypothesis, there is an automorphism $\sigma\colon H \to H$ such that $\sigma$ fixes $A_0$ pointwise and maps $B$ into $G$. Notice that $\sigma$ induces an automorphism $\hat{\sigma}$ of $\ell^{1}(H)$ via 
\begin{equation*}
    \hat{\sigma}\left( \sum_{b \in H} r_{b}\mathbf{1}_{b} \right) = \sum_{b \in H} r_{b}\mathbf{1}_{\sigma(b)}. 
\end{equation*}
Then $\hat{\sigma}(\tilde{h}) \in \ell^{1}(G)$ and by a moduli of continuity argument, 
\begin{align*}
    \varphi^{\ell^{1}(H)}(\tilde{f}_1,...,\tilde{f}_n,\tilde{h}) &= \varphi^{\ell^{1}(H)}(\tilde{f}_1,...,\tilde{f}_n,\hat{\sigma}(\tilde{h})) \\ &\approx_{\epsilon}\varphi^{\ell^{1}(H)} (f_1,...,f_n,\hat{\sigma}(\tilde{h})) < r + 3\epsilon.
\end{align*}
Since $\epsilon$ was arbitrary, $(\ell^{1}(G),*) \preceq (\ell^{1}(H),*)$ by Tarski-Vaught. 
\end{proof}

\begin{lemma}\label{lemma:next} Suppose that $G$ is $\omega$-saturated and $H$ is strongly $\aleph_0$-homogeneous. If $G \preceq H$ then $(\ell^{1}(G),*) \preceq (\ell^{1}(H),*)$. 
\end{lemma}

\begin{proof} Clearly the pair satisfies the hypothesis of Lemma \ref{lemma:auto} and thus the statement holds. 
\end{proof}

\begin{proposition}\label{prop:elementary} Suppose that $G$ and $H$ are $\omega$-saturated groups and $G \preceq H$. Then $(\ell^{1}(G),*) \preceq (\ell^{1}(H),*)$. 
\end{proposition}

\begin{proof}
Choose a model $\mathcal{U}$ such that $H \preceq \mathcal{U}$ and $\mathcal{U}$ is strongly $\aleph_0$-homogeneous. We may apply Lemma \ref{lemma:next} twice and conclude that $(\ell^{1}(H),*) \preceq (\ell^{1}(\mathcal{U}),*)$ and $(\ell^{1}(G),*) \preceq (\ell^{1}(\mathcal{U}),*)$.
Since $(\ell^{1}(G),*)$ is a substructure of $(\ell^{1}(H),*)$, we may conclude that $(\ell^{1}(G),*) \preceq (\ell^{1}(H),*)$. 
\end{proof}

\begin{theorem}\label{thm:elementaryequiv} Suppose that $G \equiv H$ and $G$ and $H$ are both $\omega$-saturated groups. Then $(\ell^{1}(G),*) \equiv (\ell^{1}(H),*)$. 
\end{theorem}

\begin{proof} Again, choose a strongly $\aleph_0$-homogeneous model $\mathcal{U}$ such that $G, H \preceq \mathcal{U}$. Then by Proposition \ref{prop:elementary}, $(\ell^{1}(G),*), (\ell^{1}(H),*) \preceq (\ell^{1}(\mathcal{U}),*)$ and so $(\ell^{1}(G),*) \equiv (\ell^{1}(\mathcal{U}),*) \equiv (\ell^{1}(H),*)$. 
\end{proof}

Finally, we prove the following corollary directly from the definability of the atoms. The result, the discrete case of the main theorems of Kawada and Wendel \cite{Ka48,We51}, which was an important mathematical inspiration for this paper. 

\begin{corollary}\label{coro:automorphism} Suppose that $\phi\colon(\ell^1(G),*)\to (\ell^1(G),*)$ is an metric isomorphism of convolution algebras. Then there is $\hat \phi\colon G\to G$ a group isomorphism that induces $\phi$. 
\end{corollary}

\begin{proof}
 Let $\mathcal{A}$ be the collection of atoms of $\ell^1(G)$. Since $\mathcal{A}$ is a definable set, $\phi(\mathcal{A})=\mathcal{A}$ and it is a bijection. Furthermore, since $(\mathcal{A},*)$ is isomorphic to $(G,\cdot)$, the map $\phi$ restricted to $(\mathcal{A},*)$ induces a isomorphism of groups $\hat \phi\colon G\to G$. The map $\hat \phi$ determines $\phi$ on $\mathcal{A}$, which can be extended linearity on $span(\mathcal{A})$ and by continuity on $\ell^1(G)$.   \end{proof}

\begin{question} Corollary \ref{coro:automorphism} gives a model-theoretic account of the results by Kawada and Wendel \cite{Ka48,We51} when the group $G$ is discrete. Can we find a model-theoretic proof of these results when $G$ is locally compact non-discrete?
\end{question}

\subsection{Failure of elementary equivalence for non-saturated models.} In this subsection, we prove that there exists discrete groups $G_1$ and $G_2$ such that $G_1 \equiv G_2$ but the associated convolution algebras are not elementary equivalent, i.e., $(\ell^{1}(G_1),*) \not \equiv ( \ell^{1}(G_2),*)$. This follows from the observation that \emph{amenability} is first-orderizable in the language of convolution algebras, but not in the language of pure groups. Keller \cite{Keller1972} introduced the notation of \emph{uniformly amenable groups}. A discrete group is uniformly amenable if every ultrapower of $G$ remains amenable. Therefore, any non-uniformly amenable group allows us to construct a pair of groups which are elementary equivalent, but that equivalence does not transfer to the associated convolution algebras.

\begin{fact}\label{fact:notequiv} There exists two discrete groups $G_1$ and $G_2$ such that $G_1 \equiv G_2$ as pure groups, but $G_1$ is amenable while $G_2$ is not amenable. 
\end{fact}

\begin{proof}
    As stated in the paragraph above, any group which is amenable but not uniformly amenable gives an example. A canonical example is the collection of permutations of $\mathbb{N}$ with finite support. This group is locally finite and thus amenable. On the other hand, it is not hard to show that any non-trivial ultrapower of this group contains copies of the free group on two generators as subgroup, and thus not amenable\footnote{This same argument is detailed by Goldbring in \cite{Goldbring2024}.}.
\end{proof}

\begin{proposition}\label{prop:amenable} Let $G$ be a discrete group. For every $n \in \mathbb{N}$, we define the sentence $\psi_{n}$ as  
\begin{equation*}
\psi_{n} := \sup_{\substack{\lVert f_1 \lVert = 1 \\ f_1 \geq 0}} \cdots\sup_{\substack{\lVert f_n \lVert = 1 \\ f_n \geq 0}} \inf_{\substack{\lVert g \lVert = 1 \\ g \geq 0}} \max \left\{ d(g, f_1 * g),\dots,d(g,f_n *g) \right\} = 0.
\end{equation*}
Then $G$ is amenable if and only if $\ell^{1}(G) \models \psi_{n}$ for every $n \geq 1$. 
\end{proposition}

\begin{proof} 
Throughout the following argument, we let $m$ denote the standard counting measure on $G$. 

Suppose that $G$ is amenable. Fix $f_1,...,f_n \in \ell^{1}(G)$ in the positive cone with norm $1$. For each $i \leq n$, there exists some element $\tilde{f}_i$ in $\ell^{1}(G)$ such that $\tilde{f}_i$ has finite support, $\tilde{f}_i$ is in the positive cone with norm $1$, and $d(f_i,\tilde{f}_i)$ is arbitrarily small. Using uniform continuity of $*$ and the distance function, it will suffice to prove that for every $\epsilon > 0$, there exists some $g$ in the positive cone with norm $1$ such that $\max_{i \leq n}\{d(g,g*\tilde{f}_i)\} < \epsilon$. For each $\tilde{f_i}$, we let $F_i = \supp(\tilde{f}_i)$ and for each $a \in \supp(\tilde{f}_i)$, we let $r_{a}^{i}$ be the weight of $a$ in $\tilde{f}_i$, i.e., $\tilde{f}_i(a)$. We let $F = \bigcup_{i \leq n} F_i$. Note that since each $\bar{f_i}$ has finite support, the set $F$ is finite. Also note that since each $\tilde{f}_i$ is in the positive cone, $r_a^i\geq 0$ and since $\Vert \tilde{f}_i\lVert=1$ we also have $\sum_{a\in F_i}r_a^i=1$.

Since $G$ is amenable, we may fix a (right) F\o lner net $(X_i)_{i \in I}$ of $G$. Choose some set $X = X_{k}$ so that for each $a \in F$, we have that 
\begin{equation*} \frac{|X \triangle (X \cdot a)|}{|X|} < \min\left\{\frac{\epsilon}{r_{c}^{i} \cdot |F_i|}: i \leq n, c \in F_{i} \right\}.
\end{equation*}
Note that if $c\in F_i=\supp(\tilde{f}_i)$ then $r_c^i>0$, so the expression above is well defined. Then let $g = \frac{1}{|X|} \mathbf{1}_{X}$. Now, notice that for each $a \in G$, we have that 
\begin{equation*}
    g*\mathbf{1}_{a}(x) = g(x \cdot a^{-1}) = \frac{1}{|X|}\mathbf{1}_{X\cdot a}. 
\end{equation*}
Then,
\begin{equation*}
    |g - g * \mathbf{1}_a| = \frac{1}{|X|} \mathbf{1}_{X \triangle (X \cdot a)}.
\end{equation*}

Therefore, by linearity of convolution,
\begin{align*}
d(g,g * \tilde{f}_i)& =d \left(\sum_{a\in F_i}r^i_a g,g*\sum_{a\in F_i}r^i_a\mathbf{1}_a \right)\\
&= \int_{G} \left|\sum_{a \in F_i} r^i_a (g - g * \mathbf{1}_{a}) \right| dm\\
&\leq \int_{G} \sum_{a \in F_i} r^i_a \left|(g - g * \mathbf{1}_{a}) \right| dm\\ 
&= \int_{G} \sum_{a \in F_i} r^i_a  \left| \left(\frac{1}{|X|} \mathbf{1}_{X \triangle (X \cdot a)}\right) \right|dm \\
&= \sum_{a \in F_i} r^i_a \left( \frac{|X \triangle (X \cdot a)|}{|X|} \right) < \sum_{a \in F_i} \frac{\epsilon}{|F_i|} = \epsilon. 
\end{align*}

For the other implication, suppose that $\ell^{1}(G) \models \psi_{n}$ for every $n \geq 1$. Let $\Gamma$ be the directed set indexed by the finite subsets of $G$ and ordered by inclusion. For each finite subset $A$ of $G$, since $G \models \psi_{|A|}$, we can find $h_{A} \in \ell^{1}(G)$ such that $\lVert h_{A} \rVert = 1$, $|\supp(h_{A})|$ is finite, and for every $a \in A$, $\lVert h_{A} * \mathbf{1}_{a} - h_{A} \rVert_{\ell^{1}(G)} < \frac{1 }{|A|}$. A straightforward computation shows that this net of finite means witnesses the amenability of $G$. 
\end{proof}

\begin{theorem}\label{thm:failure} There exist countable discrete groups $G_1$ and $G_2$ such that $G_1 \equiv G_2$ but the associated convolution algebras are not elementary equivalent, that is $(\ell^{1}(G_1),*) \not \equiv (\ell^{1}(G_2),*)$.
\end{theorem}
\begin{proof}
    Direct from Fact \ref{fact:notequiv} and Proposition \ref{prop:amenable}.
\end{proof}

The behavior of convolution algebras with respect to the underlying groups is reminiscent of the behavior of von Neumann algebras with respect to their underlying group as explained in \cite{GHT}, the algebra carries more information than the first order theory of the group.

\begin{question}
If $G$ and $H$ are $\omega$-back-and-forth-equivalent groups (see \cite{GHT} for details), are the convolution algebras $(\ell^1(G),*)$ and
 $(\ell^1(H),*)$ elementarily equivalent?
\end{question}

\section{Dividing Lines}

This section concerns convolution algebras and their connections to model-theoretic dividing lines. It is divided into two parts: the first deals with discrete groups, while the second deals more generally with locally compact groups. Several authors have made scattered remarks about dividing lines witnessed by either $(L^{1}(\mathbb{R}),*)$ or $(\ell^{1}(\mathbb{Z}),*)$. The first named author proved that $(L^1(\mathbb{R}),*)$ is unstable \cite{Be}. Subsequently, Farah, Hart, and Sherman \cite[Proposition 6.2]{FHS} proved that $(\ell^1(\mathbb{Z}),*)$ is unstable. Khanaki \cite{Khanaki} further demonstrated that $(\ell^1(\mathbb{Z}),*)$ has the strong order property and asked whether it is NIP \cite[Question 11.3]{Khanaki}. We prove that for a wide variety of groups, $(\ell^{1}(G),*)$ has $\mathrm{TP}_2$, and in particular $(\ell^{1}(\mathbb{Z}),*)$ has $\mathrm{TP}_2$, thus also $\mathrm{IP}$, resolving Khanaki's question.

Throughout the section, we mention that all the parameters involved in witnessing $\mathrm{TP}_2$ and $\mathrm{IP}$ are positive and of norm exactly $1$; that is, they are unit vectors in the positive part of the unit ball of $L^{1}(G)$. Also, we emphasize that the formulas which we use to witness the different dividing lines are all quantifier free whereas the formula in both \cite{FHS} and \cite{Khanaki} used to witness instability and the strict order property (respectively) uses an existential quantifier. 

\subsection{Discrete groups}

Throughout this subsection, all of the groups we will consider are discrete. We prove the following:

\begin{enumerate}
\item If $G$ is amenable, then the convolution algebra $(\ell^{1}(G),*)$ is unstable.
\item If $G$ admits a certain combinatorial family of subgroups, then the convolution algebra $(\ell^{1}(G),*)$ has $\mathrm{TP}_2$.
\item As a consequence, if $G$ has an infinite abelian subgroup, then the convolution algebra $(\ell^{1}(G),*)$ has $\mathrm{TP}_2$. In particular, if $G$ is solvable, then $(\ell^{1}(G),*)$ has $\mathrm{TP}_2$.
\end{enumerate}

We quickly prove statement (1) above to get some intuition on the basic tools that we will use in this section and then move on to proving statements (2) and (3). The following definition is well-known to be equivalent to the classical notion of stability in continuous logic \cite{byaacov2014model}.

\begin{definition} A theory $T$ is \textbf{unstable}, if there exists a model $M$ of $T$ and an $\mathcal{L}$-formula $\varphi(\bar{x}, \bar{y})$ with parameters $(\bar{a}_n,\bar{b}_m)_{n,m < \omega}$ from $M$ such that 
\begin{equation*}
    \lim_{n \to \infty} \lim_{m \to \infty} \varphi(\bar{a}_n,\bar{b}_{m}) \neq \lim_{m \to \infty} \lim_{n \to \infty} \varphi(\bar{a}_n,\bar{b}_{m}),
\end{equation*}
where both arrays converge. 
\end{definition}

\begin{proposition}\label{prop:unstable} Suppose that $G$ is an infinite amenable discrete group. Then $\Th(\ell^{1}(G),*)$ is unstable.
\end{proposition}
\begin{proof} Since $G$ is amenable, we may choose a net of finite subsets of $G$, $(X_i)_{i \in I}$, which witness F\o lner's criterion (Definition \ref{def:FC}). Since $G$ is infinite, for each finite subset $X$ of $G$, there exists some $a_{X} \in G$ such that $(a_{X} \cdot X) \cap X = \emptyset$. We now construct a sequence of elements in $\ell^{1}(G)$ which will show that $(\ell^{1}(G),*)$ is unstable. 

 \textbf{Step 0}: Choose an $i \in I$ and let $Y_0 = X_0$. Let $b_0 = a_{X_0}$. 

\textbf{Step n+1}: Suppose we have constructed $Y_i$ and $b_i$ for $0\leq i\leq n$. Choose $X_j$ such that for all $c \in \{b_0,...,b_n\}$, $|(c \cdot X_j) \triangle X_j| /|X_j| < \frac{1}{n+1}$. Let $Y_{n+1} = X_j$. Choose $b_{n+1}$ such that
  
 \begin{equation*} \left(b_{n+1} \cdot \bigcup_{m \leq n+1} Y_{m}\right) \cap \bigcup_{m \leq n+1} Y_{m} = \emptyset. 
 \end{equation*}

For $n < \omega$, let $f_{n} = \frac{1}{|Y_n|} \mathbf{1}_{Y_n}$ and $g_{n} = \mathbf{1}_{b_n}$. Consider the formula $\varphi(x,y):= \lVert x * y - y \rVert$. Then 
\begin{equation*}
    \lim_{m \to \infty} \lim_{n \to \infty} \varphi(g_m, f_n) = 0 \text{ while } \lim_{n \to \infty} \lim_{m \to \infty} \varphi(g_m, f_n) = 2.  
\end{equation*}
Thus $(\ell^{1}(G),*)$ is unstable. 
\end{proof}

Note that we proved a stronger statement: the function $\varphi(x,y)=\lVert x * y - y \rVert$ is not WAP with respect to the structure $(\ell^{1}(G),*)$.

We now focus on proving statements (2) and (3) from the beginning of the section. We begin by recalling the definition of the independence property, $\mathbf{IP}$ and the tree property of the second kind, $\mathbf{TP_{2}}$, in the context of continuous logic.

\begin{definition} A theory $T$ has the \textbf{independence property}, $\mathbf{IP}$, if there is a formula $\varphi(\bar{x}, \bar{y})$, positive real numbers $r$ and $\epsilon$, and for every $n < \omega$, there are parameters $a_1,...,a_n$ and $(b_{A})_{A \in \mathcal{P}([n])}$ from a model $M$ of $T$ such that for every $K\in \mathcal{P}([n])$
\[
\varphi(\bar{a}_i,\bar{b}_K) \leq r \text{ if } i \in K, \quad \text{and} \quad \varphi(\bar{a}_i,\bar{b}_K) \geq r + \epsilon \text{ if } i \in [n] \setminus K.
\]
\end{definition}

The following definition is \cite[Definition 2.4]{conant2016model}.

\begin{definition} A theory $T$ has the \textbf{tree property of the second kind}, $\mathbf{TP_{2}}$, if there is a formula $\varphi(\bar{x}, \bar{y})$, an integer $k > 0$, and an array of parameters $(\bar{a}_{i,j})_{i,j < \omega}$ from a model $M$ of $T$ such that 
\begin{enumerate}
    \item For all $\sigma \in \omega^{\omega}$, $\{\varphi(\bar{x},\bar{a}_{n,\sigma(n)}) = 0: n < \omega\}$ is satisfiable, 
    \item for all $n < \omega$, $\{\varphi(\bar{x},\bar{a}_{n,i}) = 0 : i < \omega\}$ is $k$-inconsistent. 
\end{enumerate}
\end{definition}

We quickly take the opportunity to notice that if $G$ witnesses some dividing line, then $(\ell^{1}(G),*)$ witnesses the continuous variant of the same dividing line. 

\begin{corollary} Let $G$ be a discrete group. If $G$ is unstable, $\mathrm{IP}$, or $\mathrm{TP}_2$, then $(\ell^{1}(G),*)$ is respectively unstable, $\mathrm{IP}$, or $\mathrm{TP}_2$. 
\end{corollary}

\begin{proof}
Direct from Lemmas \ref{lemma:01} and \ref{lemma:good}.
\end{proof}

The next proposition shows that if a countable group have a particular combinatorial family of subgroups, then $(\ell^{1}(G),*)$ has $\mathrm{TP}_{2}$. We will then prove a $\mathrm{TP}_{2}$ transfer theorem and work toward removing the countability assumption on $G$.

\begin{theorem}\label{prop:general}
Let $G$ be a countable discrete group. Suppose that for every natural number $k$ there exists a family of pairwise distinct (proper) amenable subgroups $(H_i)_{i \leq k}$ of $G$ and an array of elements $(b_{i,j})_{i,j \leq k}$ from $G$ such that 
\begin{enumerate}
    \item For each $\eta \colon [k] \to [k]$, $\bigcap_{i \leq k} b_{i,\eta(i)} H_{i} \neq \emptyset$. 
    \item For each $i,j,j' \leq k$, $b_{i,j} H_{i} \cap b_{i,j'} H_{i} = \emptyset$ whenever $j\neq j'$. 
\end{enumerate}
Then the formula $d(x * y, z) \dotdiv \frac{1}{2}$ has $\mathrm{TP}_{2}$. Furthermore, if one can choose the arrays of elements from $G$ such that $b_{i,0} = e$ for $i \leq k$, then $d(x *y, y)$ has $\mathrm{IP}$ relative to the theory of $(\ell^{1}(G),*)$.
\end{theorem} 

\begin{proof} Fix a non-principal ultrafilter $\mathcal{D}$ on $\mathbb{N}$ and consider the ultrapower $\prod_{\mathcal{D}} (\ell^{1}(G),*)$ which is an $\aleph_1$-saturated model of $\Th(\ell^{1}(G),*)$. We show that for every positive integer $n \geq 2$, we can find a family of parameters $(f_i, g_{i,j})_{i,j \leq n}$ from $\prod_{\mathcal{D}} \ell^{1}(G)$ such that
\begin{enumerate}
\item For each function $\eta\colon [n] \to [n]$, $\{d(x * f_i, g_{i,\eta(i)}) \dotdiv \frac{1}{2}: i \leq n\}$ is satisfiable.   
\item For each $i \leq n$, $\{d(x * f_i, g_{i,j})  \dotdiv \frac{1}{2}: j \leq n \}$ is $2$-inconsistent. 
\end{enumerate} 
By compactness, this implies the statement. First, we need to consider a family of \emph{small enough} subgroups. Fix a family of amenable (proper) subgroups $(H_{i})_{i \leq n}$ and elements $(b_{i,j})_{i,j \leq n}$ of $G$ with the condition described in the statement of the theorem. 

Now we construct the parameters $(f_{i},g_{i,j})_{i,j \leq n}$ in $\prod_{D} (\ell^{1}(G),*)$. The parameter $f_i$ will correspond essentially to a \emph{mean} on $H_i$ and $g_{i,j}$ will correspond to a translation of $H_{i}$. Since $H_i$ is amenable and countable, we may choose a F\o lner sequence $(X_{i,m})_{m < \omega}$ for $H_i$. We let $f_{i} = [\frac{1}{|X_{i,m}|}\mathbf{1}_{X_{i,m}}]_{\mathcal{D}}$. We let $g_{i,j} = \mathbf{1}_{b_{i,j}} * f_{i}$ [formally speaking, $\mathbf{1}_{b_{i,j}} = \Delta(\mathbf{1}_{b_{i,j}})$ where $\Delta\colon G \to \prod_{\mathcal{D}} G$ is the diagonal embedding]. It follows from the definition of convolutions that $g_{i,j} = [\frac{1}{|X_{i,m}|}\mathbf{1}_{(b_{i,j}\cdot X_{i,m})}]_{\mathcal{D}}$. 

We prove the claims stated at the beginning of the proof, namely the $2$-inconsistency of rows and the consistency of paths. To show that rows are $2$-inconsistent, notice that $d(g_{i,j},g_{i,j'}) = 2$ whenever $j' \neq j$. Hence, if \begin{equation*} 
d(x * f_{i}, g_{i,j}) \leq \frac{1}{2} \text{ and } d(x * f_{i}, g_{i,j'}) \leq \frac{1}{2},
\end{equation*} 
is consistent, then by the triangle inequality $d(g_{i,j},g_{i,j'}) \leq 1$, a contradiction. To see that paths are consistent, choose any map $\eta\colon [n] \to [n]$. By the hypothesis in our statement, there exists an element $c \in G$ such that $c \in \bigcap_{i \leq n} b_{i,\eta(i)} \cdot H_i$. Thus, for each $i \leq n$ we may write $c = b_{i,\eta(i)} \cdot d_i$ where $d_i \in H_{i}$. Then 
\begin{align*} 
d(\mathbf{1}_{c} * f_{i},g_{i,\eta(i)}) = \left[\frac{|c \cdot X_{i,m} \triangle b_{i,\eta(i)} \cdot X_{i,m}|}{|X_{i,m}|} \right]_{D} &=  \left[\frac{|b_{i,\eta(i)} \cdot d_i \cdot X_{i,m} \triangle b_{i,\eta(i)} \cdot X_{i,m}|}{|X_{i,m}|} \right]_{D} \\
&= \left[\frac{| d_i \cdot X_{i,m} \triangle  X_{i,m}|}{|X_{i,m}|} \right]_{D} =0.
\end{align*} 

The third equality follows by multiplying both sides of the symmetric difference by $b_{i,\eta(i)}^{-1}$. The fourth equality follows from the fact that as $m$ goes to infinity, the value of the term $(|d_i \cdot X_{i,m} \triangle X_{i,m}|/|X_{i,m}|)$ goes to $0$, since $d_{i} \in H_{i}$ and $(X_{i,m})_{m < \omega}$ is a F\o lner sequence. 

To show the moreover portion of the statements, i.e., that $d(x * y, y)$ has $\mathrm{IP}$ with the additional assumption on the parameters, we argue that one can shatter the set $\{f_{i}: i \leq n\}$. Indeed, for $A \subseteq [n]$, consider $\mathbf{1}_{a_{A}}$ where $a_A \in \bigcap_{i = 1}^{n} \alpha_i H_i$, $\alpha_i = b_{i,0}=e$ if $i \in A$, otherwise $\alpha_i = b_{i,1}$.  Then if $i \in A$, $d(\mathbf{1}_{a_{A}} * f_i,f_i) = 0$ and if $i \not \in A$, $d(\mathbf{1}_{a_{A}} * f_i,f_i) = 2$. 
\end{proof}

In general, $\mathrm{TP}_2$ does not transfer from substructures to superstructures via embeddings. However, we will see that this is the case for the particular formula involved in the proof of Theorem \ref{prop:general}.

\begin{lemma}\label{lemma:TP2} Fix a discrete group $G$. Suppose that there exists a model $N$ of $\Th(\ell^{1}(G),*)$ and an array of parameters $(a_i,b_{i,j})_{i,j < \omega}$ from $N$ such that 
\begin{enumerate}
    \item The formulas $\{d(x * a_i, b_{i,j}) \dotdiv \frac{1}{2}: i, j < \omega\}$ witnesses $\mathrm{TP}_2$. 
    \item There exists $\epsilon > 0$ so that for each $i, j , k < \omega$, $j \neq k$, $d(b_{i,j},b_{i,k}) \geq 1 + \epsilon$. 
\end{enumerate}
Then if there is a metric embedding from the convolution algebra $(\ell^{1}(G),*)$ into an  $\mathcal{L}_{\cnv}$-structure $M$, then the formula $\varphi(x;y,z) = d(x * y,z) \dotdiv \frac{1}{2}$ has $\mathrm{TP}_{2}$ with respect to $\Th(M)$.
\end{lemma}

\begin{proof} Identify the metric embedding of $(\ell^{1}(G),*)$ in $M$ with $(\ell^{1}(G),*)$. After taking a metric ultrapower of $M$, namely $\prod_{D} M$, we let $N'$ be the substructure corresponding to the metric ultrapower of the embedded copies of $(\ell^{1}(G),*)$. By saturation, we may find an array of parameters in $N'$ with properties $(1)$ and $(2)$ as described in the lemma. Since $\varphi(x;y,z)$ is quantifier free, the paths (which are satisfiable in $N'$) remain satisfiable (in $\prod_{D}M$). So it suffices to prove that the rows are 2-inconsistent. Fix $j \neq k$ and suppose there is some $h \in \prod_{D} M$ such that
\begin{equation*}
    d(h * a_i, b_{i,j}) \leq \frac{1}{2} \text{ and } d(h * a_i, b_{i,k}) \leq \frac{1}{2}. 
\end{equation*}
Then by the triangle inequality, 
\begin{equation*}
    d(b_{i,j},b_{i,k}) \leq d(h * a_i, b_{i,j}) + d(h * a_i, b_{i,k}) \leq \frac{1}{2} + \frac{1}{2} = 1.  
\end{equation*}
This contradicts the fact that  $d(b_{i,j},b_{i,k}) \geq 1 + \epsilon$.
\end{proof}

The next theorem shows that the countability assumption can be dropped from Theorem \ref{prop:general}.

\begin{theorem}\label{thm:general}
Let $G$ be a discrete group. Suppose that for every natural number $k$ there exists a family of pairwise distinct (proper) amenable subgroups $(H_i)_{i \leq k}$ of $G$ and an array of elements $(b_{i,j})_{i,j \leq k}$ from $G$ such that 
\begin{enumerate}
    \item For each $\eta \colon [k] \to [k]$, $\bigcap_{i \leq k} b_{i,\eta(i)} H_{i} \neq \emptyset$. 
    \item For each $i,j,j' \leq k$, $b_{i,j} H_{i} \cap b_{i,j'} H_{i} = \emptyset$ whenever $j\neq j'$. 
\end{enumerate}
Then the formula $d(x * y, z) \dotdiv \frac{1}{2}$ has $\mathrm{TP}_{2}$. Furthermore, if one can choose the arrays of elements from $G$ such that $b_{i,0} = e$ for $i \leq k$, then $d(x *y, y)$ has $\mathrm{IP}$ relative to the theory of $(\ell^{1}(G),*)$.
\end{theorem}

\begin{proof} Consider the countable language $\mathcal{L}' = \mathcal{L}_{\grp} \cup \{(H_{k,i}(x))_{i \leq k}: k < \omega\}$. $G$ can naturally be extended to a structure in this language by interpreting each finite sequence of unary predicates $(H_{k,i}(x))_{i \leq k}$ as a finite family of amenable subgroups of $G$ with the appropriate intersection property stipulated in the hypothesis. By the downward L\"{o}wenheim-Skolem theorem, there exists a countable $\mathcal{L}'$-structure $G_0$ such that $G_0 \preceq G$. Notice that for every $k < \omega$ and $i \leq k$, $H_{k,i}(G_0)$ is a countable subgroup of $H_{k,i}(G)$. Since $H_{k,i}(G)$ is amenable, so is $H_{k,i}(G_0)$. Since $G_0 \preceq G$, the sequence of subgroups $(H_{k,i}(G_0))_{i \leq k}$ gives a family of amenable subgroups of $G_0$ with the appropriate intersection property, i.e., there exists parameters $(b_{i,j})_{i,j \leq k}$ in $G_0$ such that $(H_{k,i}(G_0))_{i \leq k}$ and $(b_{i,j})_{i,j \leq k}$ witness the hypothesis of Theorem \ref{prop:general}. Moreover, the inclusion map $\iota\colon (\ell^{1}(G_0),*) \to (\ell^{1}(G),*)$ is clearly a metric embedding that preserves the convolution operation and thus by Lemma \ref{lemma:TP2}, $d(x * y,z) \dotdiv \frac{1}{2}$ has $\mathrm{TP}_{2}$ with respect to $\Th(\ell^{1}(G),*)$. Finally, assuming the moreover condition, we remark that $d(x * y, y)$ has $\mathrm{IP}$ with respect to $\Th(\ell^{1}(G),*)$ since $(\ell^{1}(G_0),*)$ is a substructure of $(\ell^{1}(G),*)$, $d(x * y, y)$ has $\mathrm{IP}$ with respect to this model, and $d(x * y,y)$ is a quantifier free formula.
\end{proof}

Our goal now is to show that $\Th(\ell^{1}(G),*)$ has $\mathrm{TP}_2$ for a wide range of concrete groups. In particular, we prove that if $G$ contains an infinite abelian subgroup, then $\Th(\ell^{1}(G),*)$ has $\mathrm{TP}_2$. To do so, we require a series of lemmas, beginning with the following classical description of  torsion abelian groups.

\begin{fact}\cite[Theorem 8.4]{Fuchs1970}\label{Fuch} A torsion abelian group $A$ is the direct sum of $p$-groups $A_{p}$ belonging to different primes $p$. The $A_{p}$ are uniquely determined by $A$. 
\end{fact}

As a consequence of this theorem, any infinite torsion group must contain either finite groups that split into products of arbitrarily many components, or cyclic subgroups $\mathbb{Z}/n\mathbb{Z}$ with $n$ arbitrarily large. We make this precise in the following fact.

\begin{fact}\label{fact:subgroup} Suppose that $G$ is a torsion group that contains an infinite abelian subgroup. Then either 
\begin{enumerate}
    \item There exists arbitrarily large $n$ such that $H_1 \times ... \times H_{n}$ is a subgroup of $G$ and each $H_i$ for $i \leq n$ is a finite non-trivial abelian group. 
    \item There exists a prime $p$ such that for every integer $m$, there exists an integer $n > m$ with $\mathbb{Z}/p^{n}\mathbb{Z}$ a subgroup of $G$. 
\end{enumerate}
\end{fact}

\begin{proof} Let $H$ be an infinite abelian subgroup of $G$. Since $G$ is torsion, so is $H$. By Fact \ref{Fuch}, $H$ is a direct sum of $p$-groups, say 

\begin{equation*}
    H \cong \bigoplus_{p} A_{p}.
\end{equation*}

If there are infinitely many primes for with which the component $A_{p}$ is non-trivial, we are in case (1), since every torsion group clearly contains a finite subgroup. Thus we may assume that there are only finitely many non-trivial components.  By the pigeon hole principle, there exists some $p$ such that $|A_{p}|$ is infinite. For  integer $n$, consider $(p^{n})^{n}$ many distinct point in $A_{p}$ and consider the generated subgroup, $K$. By the fundamental theorem of finitely generated abelian groups, we have that $K$ is isomorphic to $\prod_{i \leq m} \mathbb{Z}/p^{k_{i}}\mathbb{Z}$ for some integers $k_{1},...,k_{m}$. If $m \geq n$, we have satisfied condition (1). Now assume that $m < n$. Since our finite group is generated by $(p^{n})^{n}$-many elements, the group must be at least of cardinality $(p^{n})^{n}$. Thus, $ (p^{n})^{n} \leq |K| = \prod_{i \leq m } p^{k_i} \leq \left( p^{\max_{i \leq m}\{k_i\}} \right)^{m}$. Since $m < n$, $\max_{i \leq m} \{k_i\} > n$. 

Thus, for every integer $n$, we have found a subgroup $K$ such that either $K$ is the product of at least $n$-non trivial abelian groups or $K$ contains a $p$-group of size at least $p^{n}$. This completes the proof.  
\end{proof}

We will now show that the convolution algebra $(\ell^{1}(\mathbb{Z}),*)$ embeds into an ultraproduct of convolution algebras over finite groups. The argument is straightforward and relies on two observations. First, finitely supported functions are dense in $\ell^{1}(\mathbb{Z})$; hence, every element of $\ell^{1}(\mathbb{Z})$ can be approximated arbitrarily well in the $\ell^{1}$-norm by finitely supported functions. Second, finitely supported functions in $\ell^{1}(\mathbb{Z})$ are closed under the operations in our language. In particular, these operations only depend on finitely many coordinates and therefore take place entirely within some large finite interval of $\mathbb{Z}$. Such finite pieces can be faithfully modeled inside $\mathbb{Z}/n\mathbb{Z}$ for sufficiently large $n$. Combining these observations, one obtains an embedding of $(\ell^{1}(\mathbb{Z}),*)$ into an appropriate ultraproduct of the convolution algebras $(\ell^{1}(\mathbb{Z}/n\mathbb{Z}),*)$. More interestingly, in the next section we will extend this construction to show how $(L^{1}(G),*)$ embeds into similar ultraproducts when $G$ is a connected abelian Lie group.

\begin{proposition}\label{prop:discrete} Suppose that $(n_i)_{i < \omega}$ is a sequence of strictly increasing positive integers. Consider each $\mathbb{Z}/n_i\mathbb{Z}$ as a group with the counting measure and $D$ a non-principal ultrafilter on $\mathbb{N}$. Then there is a metric embedding of convolution algebras from $(\ell^{1}(\mathbb{Z}),*)$ into $\prod_{D} (\ell^{1}(\mathbb{Z}/n_i\mathbb{Z}),*)$. 
\end{proposition}

\begin{proof} The collection of finitely supported functions on $\mathbb{Z}$, $c_{00}(\mathbb{Z})$, is a dense subset of $\ell^{1}(\mathbb{Z})$. It suffices to define our metric embedding for the subspace $c_{00}(\mathbb{Z})$ and extend by continuity. We define the embedding $\Phi$ as follows: If $f \in c_{00}(\mathbb{Z})$ then we let $\Phi(f) = [\hat{f}_i]_{D}$ where 
\begin{enumerate}
    \item If $\supp(f)$ is not a subset of the interval $[-n_i+,n_i]$, then $\hat{f}_{i}$ is $0$. 
\item If $\supp(f) \subseteq [-n_i,n_i]$,
then 
$f = \sum_{-n_i \leq k \leq n_i} r_k \mathbf{1}_{k}$ and we set
\begin{equation*}
    \hat{f}_i = \sum_{-n_i \leq k \leq n_i} r_k \mathbf{1}_{[k]_i}, 
\end{equation*}
where $[k]_i$ is the equivalence class of the integer $k$ in $\mathbb{Z}/n_i\mathbb{Z}$. 
\end{enumerate}
We check that $\Phi$ respects the norm and convolution and leave the other operations (i.e., addition, min, max) as an exercise. 
\begin{claim-star} The map $\Phi$ preserves the norm. 
\end{claim-star}
\begin{claimproof}

Fix $f \in c_{00}(\mathbb{Z})$ and fix $i < \omega$ such that $\supp(f) \subseteq [-n_i, n_i]$. Fix $j < \omega$ such that $n_j > 2n_i$. Then for any $t \geq j$,  if $f = \sum_{-n_i \leq k \leq n_i} r_k\mathbf{1}_{k}$, the map $\iota_{t}\colon [-n_i,n_i] \to \mathbb{Z}/n_t\mathbb{Z}$ via $a \mapsto [a]_{t}$ is injective and so,
\begin{equation*}
    \lVert \hat{f}_t \rVert_{\ell^1(\mathbb{Z}/n_t\mathbb{Z})} =  \sum_{-n_i \leq k \leq n_i} |r_k| = \lVert f \rVert_{\ell^1(\mathbb{Z})} .
\end{equation*}
Hence this equality holds for cofinitely many indices. Since $D$ is non-principal, the conclusion holds. 
\end{claimproof}
\begin{claim-star} The map $\Phi$ preserves convolution. 
\end{claim-star}

\begin{claimproof}

Fix $f, g \in c_{00}(\mathbb{Z})$ and let $h = f * g$. Fix $i < \omega$ such that $\supp(f),\supp(g),\supp(f) + \supp(g) \subset [-n_i,n_i]$. Then $\supp(h) \subset [-n_i,n_i]$. Fix $j < \omega$ such that $n_{j} > 2n_i$. It suffices to show that for $t > j$,
\begin{equation*}
    \lim_{t \to \infty} d(\hat{f}_t * \hat{g}_{t}, \hat{h}_t) = 0. 
\end{equation*}

If $f = \sum_{-n_i \leq a \leq n_i} r_a \mathbf{1}_{a}$ and $g = \sum_{-n_i \leq b \leq n_i} s_b \mathbf{1}_{b} $, then again the map $\iota_{t}\colon [-n_i,n_i] \to \mathbb{Z}/n_t\mathbb{Z}$ is injective and so,
\begin{align*}
    \hat{f}_{t} * \hat{g}_{t} &= \left( \sum_{-n_i \leq a \leq n_i} r_a \mathbf{1}_{[a]_{t}} \right) * \left( \sum_{-n_i \leq b \leq n_i} s_b \mathbf{1}_{[b]_{t}} \right) \\ 
    &=\sum_{-n_i \leq a,b \leq n_i} r_a s_b \mathbf{1}_{[a + b]_{t}} \\
    &= \hat{h}_t.
\end{align*}
The last equality follows from the definition of convolution and the fact that $\supp(h) \subset [-n_i,n_i]$.
\end{claimproof}
\end{proof}

We now prove that if $G$ contains an infinite abelian group, then the theory of the convolution algebra $(\ell^{1}(G),*)$ has $\mathrm{TP}_{2}$. 

\begin{corollary}\label{cor:tp2} For the following groups $G$, the formula $d(x * y,z)\dotdiv\frac{1}{2}$ has $\mathrm{TP}_2$ and $d(x * y, y)$ has $\mathrm{IP}$ relative to $\Th(\ell^1(G),*)$. 
\begin{enumerate} 
\item $G = \mathbb{Z}$. 
\item $G = \bigoplus_{i \leq \kappa} G_i$ or $G = \prod_{i \leq \kappa} G_i$, where each $G_i$ is a nontrivial amenable group and $\kappa\geq \aleph_0$. 
\item $G$ contains an infinite abelian subgroup. 
\end{enumerate} 
\end{corollary}

\begin{proof} We apply Theorem \ref{thm:general} to the examples above. 
\begin{enumerate} 
\item 

For fixed $k$, let $p_0,...,p_k$ be distinct prime numbers which are strictly greater than $k$. Let $H_i = \mathbb{Z}/p_{i}\mathbb{Z}$ and let $b_{i,j} = j$. This family of groups and elements satisfy the hypothesis of Theorem \ref{thm:general}. 
\item 

Fix $k$. First we construct some \emph{large enough} subgroups by considering finite products from our indexed family. We define them recursively. First, we set $M_0 = G_0 \times \dots \times G_{t_0}$ where $t_0$ is the smallest index such that $|G_0 \times \dots \times G_{t_0}| > k$. Suppose we have constructed $M_{i}\leq G$ and $t_i$ for $i\leq j$ in such a way so that $M_{i}$ is a finite product of the $G_m$'s (which we identify with its canonical image in $G$), $|M_i| > k$ and $M_i\cap M_\ell=\{e\}$ for $i < \ell \leq j$. Then $M_{j+1} = G_{t_{j}+1} \times \dots \times G_{t_{j+1}}$ (which again we identify with its canonical image in $G$) where $t_{j+1}$ is the smallest index such that $|G_{t_{j}+1} \times \dots \times G_{t_{j+1}}| > k$. For each $i \leq k$, we let 
\begin{equation*}
    H_i=\prod_{\substack{j\leq k \\ j\neq i}}M_j, 
\end{equation*} $b_{i,0} = e$, and $b_{i,1},...,b_{i,k}$ be a sequence of distinct non-identity elements from $M_i$. We identify $(H_{i})_{i \leq k}$ and $(b_{i,j})_{i,j \leq k}$ with their corresponding images under the obvious inclusion maps into $G$. It is clear by construction that if $j < j' \leq k$ then $b_{i,j} \cdot H_i \cap b_{i,j'} \cdot H_i = \emptyset$. Moreover, for each $\eta\colon [k] \to [k]$, we have $b_{0,\eta(0)} \cdot ... \cdot b_{k,\eta(k)}  \in \bigcap_{i \leq k} b_{i,\eta(i)} \cdot H_{i}$. 
\item If $G$ contains an element without torsion, then $G$ contains a copy of $\mathbb{Z}$ and thus we may apply Theorem \ref{thm:general} and the same argument from (1). Otherwise, $G$ is a torsion group with an infinite torsion abelian subgroup. By Fact \ref{fact:subgroup}, we may find certain collections of finite subgroups, either a prime $p$ and a strictly increasing sequence of integers $(m_i)_{i < \omega}$ so that $G$ contains subgroups of the form $\mathbb{Z}/p^{m_i}\mathbb{Z}$ or finite products of finite abelian groups with arbitrarily many non-trivial components. 
\begin{enumerate}[(a)]
    \item First, suppose $G$ contains a sequence of subgroups of the form $\mathbb{Z}/p^{m_i}\mathbb{Z}$ for a strictly increasing sequence of integers $(m_i)_{i < \omega}$. There is an embedding from $(\ell^{1}(\mathbb{Z}),*)$ to $\prod_{D}(\ell^{1}(\mathbb{Z}/p^{m_i}\mathbb{Z}),*)$ (Proposition \ref{prop:discrete}), which is a substructure of $\prod_{D} (\ell^{1}(G),*)$. By Lemma \ref{lemma:TP2}, the existence of this embedding implies that $d(x * y, z) \dotdiv \frac{1}{2}$ has $\mathrm{TP_2}$ with respect to $\Th(\ell^{1}(G),*))$. Note that $d(x * y, y)$ has $\mathrm{IP}$ with respect to $\Th(\ell^{1}(G),*)$ since it is a quantifier free formula which has $\mathrm{IP}$ with respect to a substructure (i.e., the embedded copy of $(\ell^{1}(\mathbb{Z}),*)$).
    \item Now we suppose that $G$ has finite products of finite abelian groups with arbitrarily many components. Then the statement follows directly from Theorem \ref{thm:general} and the argument from case (2).   \qedhere
\end{enumerate}
\end{enumerate} 
\end{proof}

We now give an elementary argument to show that if $G$ is an infinite solvable group, then $(\ell^{1}(G),*)$ has $\mathrm{TP}_{2}$. To do so, we need a quick lemma.

\begin{lemma}\label{lemma:quotient group}
Let $G$ be a discrete group with a finite normal subgroup $H$. Let $K$ denote the quotient group $G/H$. Then $(\ell^1(K),*)$ metrically embeds into $(\ell^1(G),*)$.
\end{lemma}
\begin{proof}
Let $\pi\colon G\to K=G/H$ denote the quotient map.  
For every $f\in \ell^1(K)$, define $\tilde{f}\colon G\to\mathbb{R}$ by $\tilde{f}(x)=\frac{f(\pi(x))}{|H|}$, for $x\in G$. 
It is straightforward to verify that $\tilde{f}\in\ell^1(G)$ and furthermore $\lVert f \rVert_{\ell^1(H)}= \lVert \tilde{f} \rVert_{\ell^1(G)}$. We verify that the map sending $f$ to $\tilde{f}$ preserves convolution. 

\begin{claim-star} For all $f,g\in\ell^1(K)$ we have  that $\tilde{f}*\tilde{g}=\widetilde{f*g}$.
\end{claim-star}
\begin{claimproof}

Fix $x \in G$ and let us show that $(\tilde{f}*\tilde{g})(x)=(\widetilde{f*g})(x)$. On one side, we have 

\begin{equation*}
    (\widetilde{f*g})(x)=\frac{(f*g)(\pi(x))}{|H|}=\frac{1}{|H|}\sum_{k\in K}f(k)g(k^{-1}\pi(x)).
\end{equation*}
On the other hand,
\begin{align*}
    (\tilde{f}*\tilde{g})(x)=\frac{1}{|H|^2}\sum_{y\in G}f(\pi(y))g(\pi(y^{-1}x))=\frac{1}{|H|^2}\sum_{y\in G}f(\pi(y))g(\pi(y^{-1})\pi(x)).
\end{align*}
For each $k\in K$, the set $\pi^{-1}(k)$ has $|H|$-many preimages and for $y_1,y_2\in \pi^{-1}(k)$ we have $\tilde{f}(y_1)=\tilde{f}(y_2)=\frac{f(k)}{|H|}$, thus $\sum_{y\in \pi^{-1}(k)}\tilde{f}(y)=f(k)=\frac{\sum_{y\in \pi^{-1}(k)}f(k)}{|H|}$. Thus,
\begin{align*}(\tilde{f}*\tilde{g})(x) &=\frac{1}{|H|^2}\sum_{y\in G}f(\pi(y))g(\pi(y^{-1})\pi(x)) \\ 
&=\frac{1}{|H|^2}\sum_{k\in K}\sum_{y\in \pi^{-1}(k)}f(\pi(y))g(\pi(y^{-1})\pi(x)) \\
&=\frac{1}{|H|}\sum_{k\in K}\sum_{y\in \pi^{-1}(k)}\frac{1}{|H|}f(k)g(k^{-1}\pi(x)) \\
&=\frac{1}{|H|}\sum_{k\in K}f(k)g(k^{-1}\pi(x))=(\widetilde{f*g})(x).
\end{align*}
\end{claimproof}

Thus, $(\ell^1(K),*)$ is a substructure in $(\ell^1(G),*)$ as a convolution algebra.
\end{proof}

\begin{proposition}\label{prop:solvable} If $G$ is an infinite solvable group, then the formula $d(x * y,z)\dotdiv\frac{1}{2}$ had $\mathrm{TP}_2$ and $d(x * y, y)$ has $\mathrm{IP}$ relative to $\Th(\ell^{1}(G),*)$. 
\end{proposition} 

\begin{proof} Using Lemma \ref{lemma:quotient group}, we prove that there exists an infinite abelian group $H$ such that $(\ell^{1}(H),*)$ metrically embeds into $(\ell^{1}(G),*)$. By Lemma \ref{lemma:TP2} and Corollary \ref{cor:tp2}, we will conclude that $\Th(\ell^{1}(G),*)$ witnesses the appropriate dividing lines ($\mathrm{TP}_2$ and $\mathrm{IP}$ for the respective formulas).

The derived series of $G$ is defined as follows:
$$G\trianglerighteq G'\trianglerighteq G''\cdots\trianglerighteq G^{(n)}\trianglerighteq\cdots,$$
where $G'=[G,G]$ is the commutator subgroups of $G$, and, inductively, we set $G^{(n)}=[G^{(n-1)},G^{(n-1)}]$ for each $n\in \mathbb{N}$. Also, each quotient group $G^{(n)}/G^{(n+1)}$ is abelian for each $n\in \mathbb{N}$. Since $G$ is solvable, $G^{(n)}=1$ for some $n\in \mathbb{N}$. 
If $G^{(n-1)}$ is infinite, then it is a countable abelian group and thus $G$ contains an infinite normal subgroup. By (3) of Corollary \ref{cor:tp2}, the conclusion of the Proposition holds. On the other hand, if $G^{(n-1)}$ is finite, then there is $k>1$ such that $G^{(n-k)}$ is infinite and $G^{(n-i)}$ is finite for each $1\leq i\leq k-1$. Then $G^{(n-k-1)}$ is a finite normal subgroup of $G^{(n-k)}$, and $G^{(n-k)}/G^{(n-k-1)}$ is an infinite abelian group, which we will denote by $K$. By Lemma \ref{lemma:quotient group}, $(\ell^1(K),*)$ is a substructure inside the convolution algebra $(\ell^1(G^{(n-k)}),*)$ which is itself a substructure of $(\ell^1(G),*)$. By Lemma \ref{lemma:TP2} and Corollary \ref{cor:tp2}, we conclude that the conclusion of the Proposition holds. 
\end{proof}

\begin{remark}  It was pointed out to us by Jinhe Ye that every infinite solvable group contains an infinite abelian subgroup and thus Corollary \ref{cor:tp2} applies. We thank him for allowing us to include his argument in our article. 
\end{remark}

We now give a proof that every infinite solvable group contains an infinite abelian subgroup. We first need some definitions. An \textbf{FC-group} is a group such that every element has only finitely many conjugates. An FC-group $G$ is a \textbf{BFC-group} if there exists $n \in \mathbb{N}$ such that for each $b \in G$, its conjugacy class $\{g^{-1}bg : g \in G\}$ has size bounded by $n$. The following is a characterization of such groups due to Neumann \cite{neumann1954groups}.

\begin{fact}\label{Neumann}
A group $G$ is a $\mathrm{BFC}$-group if and only if $G' = [G, G]$ is finite.
\end{fact}

\begin{proof}[Alternative proof of Proposition \ref{prop:solvable} due to Jinhe Ye]

By Corollary \ref{cor:tp2}, it suffices to prove that if $G$ is an infinite solvable group, then $G$ contains an infinite abelian subgroup. Consider the derived series of $G$. Then there must exist an index  $i$ such that $G_i$ is  infinite and $G_{i+1} = [G_i, G_i]$ is finite. We will construct an infinite abelian subgroup of $G_i$. To make things easier, we use $G, G'$ to denote $G_i, G_{i+1}$ respectively. By Fact \ref{Neumann}, $G$ is $\mathrm{BFC}$.

Suppose that $G$ does not contain any infinite abelian subgroup. Then $G$ must contain a maximal finite abelian subgroup. Otherwise, there exists a chain of finite abelian subgroups whose union is an infinite abelian subgroup, which is a contradiction. Let $A$ be a maximal finite abelian subgroup of $G$. By the fact that $G$ is $\mathrm{BFC}$, for each $x \in G$, $C_G(x) = \{g \in G : gx = xg\}$ has finite index in $G$, since its index is the size of the conjugacy class $\{g^{-1}xg : g \in G\}$.
Since $A$ is finite, its centralizer $C_G(A) = \bigcap_{a \in A} C_G(a)$ is a finite intersection of subgroups of $G$ of finite index, and thus $[G : C_G(A)]$ is also finite. In particular, $C_G(A)$ is infinite.
Take $x \in C_G(A)\backslash A$. Since $x$ commutes with every element in $A$, the subgroup generated by $A$ and $x$ is abelian. This contradicts the maximality of $A$. 
\end{proof}

We end this subsection with some open questions. 

\begin{question} Suppose that $G$ is an infinite amenable group. Does the formula $d(x * y, z) \dotdiv \frac{1}{2}$ have $\mathrm{TP}_{2}$ relative to the theory of the convolution algebra $(\ell^{1}(G),*)$? We remark that if every infinite amenable group contains an infinite abelian group, then the statement is true by Corollary \ref{cor:tp2}. To our knowledge, this question is still open (see e.g., \cite[Question 11.1]{Bergelson2008QuestionsAmenability}). 
\end{question} 

\begin{question}\label{Question:Tarski} Are there any infinite groups such that the theory of the convolution algebra $(\ell^{1}(G),*)$ is $\mathrm{NTP}_{2}$? Suppose that $G$ is a Tarski monster, i.e., $G$ is an infinite $p$-group such that any pair of non-commuting elements generate $G$. Does the formula $d(x * y, z) \dotdiv \frac{1}{2}$ have $\mathrm{TP}_{2}$ relative to the theory of the convolution algebra $(\ell^{1}(G),*)$? Does $d(x *y, y)$ have $\mathrm{IP}$? In this case, the obstacle is that $G$ does not have any infinite amenable subgroups.
\end{question} 

\subsection{Locally compact groups} In this section, we prove that if $G$ is locally compact and admits elements of arbitrarily large finite order, then the formulas from the previous subsection have $\mathrm{TP}_{2}$ and $\mathrm{IP}$ with respect to $Th(L^{1}(G),*)$. The parameters are different than the ones presented in the previous section. The idea is the following: For every natural number $n$, we find a small neighborhood of the identity $U$ and an element $a$ such that the translates $U,aU,a^{2}U,...,a^{n}U$ are all pairwise disjoint. This configuration will encode a large enough portion of the group of integers so that we may witness $\mathrm{TP_{2}}$ in a similar fashion to the one we used in previous subsection. We also need to work with an approximate identity, yet this complexity will be smoothed out by working over ultrapowers. 

The following proposition is a simplified version of a statement from Folland's text on Harmonic Analysis \cite[Proposition 2.44]{folland2016harmonic}. 

\begin{proposition}\label{prop:approx_id} Fix $G$ a locally compact group. Let $\mathcal{O}$ be a neighborhood base at $e$ in $G$. For each $U \in \mathcal{O}$, let $\psi_{U}$ be a function such that 
\begin{enumerate}
    \item $\supp (\psi_{U})$ is compact and contained in $U$. 
    \item $\psi_{U}(x) = \psi_{U}(x^{-1})$ for every $x \in G$.
    \item $\psi_{U} \geq 0$ and $\int_{G} \psi_{U} d\mu = 1$. 
\end{enumerate}
Then for every $f \in L^{1}(G)$ we have both 
\begin{equation*}
    \lim_{U \in \mathcal{O}} \lVert \psi_{U} * f - f \rVert_{L^{1}(G)} = 0, 
\end{equation*}
and, 
\begin{equation*}
    \lim_{U \in \mathcal{O}} \lVert f *  \psi_{U} - f \rVert_{L^{1}(G)} = 0. 
\end{equation*} 
We call $(\psi_{U})_{U \in \mathcal{O}}$ an \textbf{approximate identity}. Since $\lVert \psi_{U}\lVert_{L^{1}(G)}=1$ it is also called a \textbf{bounded approximate identity} or just \textbf{b.a.i.} 
\end{proposition}

The following fact is elementary and follows directly from an exercise in Munkres's text on point-set topology \cite[Page 146, Exercise 7(b)]{Munkres}. 

\begin{fact}\label{fact:translate} Let $G$ be a Hausdorff topological  group and $g_1,...,g_n$ be distinct elements in $G$. Then there exists an open neighborhood $V$ of the identity such that for each $i \neq j$, $g_i V \cap g_j V = \emptyset$. 
\end{fact}

\begin{proof} Since the sequence of elements $g_1,...,g_n$ are all distinct, we have that $g_i^{-1} g_j \neq e$. Since the group is Hausdorff, there exists an open neighborhood $U$ of $e$ that does not contain $\{g_i^{-1}g_j: 1 \leq i \neq j \leq n\}$. Since multiplication and inversion are both continuous, the map $f\colon G\times G \to G$ defined by $f(x,y) = x \cdot y^{-1}$ is also continuous. Since $f$ is continuous at $e$, we can find open sets $ W_1,W_2$ containing the identity such that $f(W_1,W_2) \subseteq U$. Let $V = W_1 \cap W_2$, then $V$ is an open set around $e$ such that $V \cdot V^{-1} \subseteq U$.

We claim that for each $i \neq j$ we have $g_i V \cap g_j V = \emptyset$. Suppose not, then there exists $a,b \in V$ such that $g_j a = g_i b$. So, $g_i^{-1} g_j = b a^{-1} \in V \cdot V^{-1} \subseteq U$. This is a contradiction. 
\end{proof}

\begin{definition} Let $G$ be a locally compact group. Then $G$ acts on $L^{1}(G)$ both on the left and on the right. For $a \in G$ and $f \in L^{1}(G)$, we let $(a \cdot f)(x) = f(a^{-1} \cdot x)$ and $(f \cdot a)(x) = f(x \cdot a^{-1})$. 
\end{definition}

\begin{fact}\label{fact:compute} Let $G$ be a locally compact group and suppose that $f,h \in L^{1}(G)$ and $a \in G$. Then $(f \cdot a) * h = f * (a \cdot h)$.  
\end{fact} 

\begin{proof} Let $f,h,g$ be as above. Fix $x \in G$. Then 
\begin{align*} 
((f \cdot a) * h)(x) &= \int_{t \in G} f(t \cdot a^{-1})h(t^{-1} \cdot x) d\mu \\ 
&\overset{s = t \cdot a^{-1}}{=} \int_{s \in G} f(s) h(a^{-1} \cdot s^{-1} \cdot x) d\mu \\ 
&= (f * ( a \cdot h))(x). \qedhere
\end{align*} 
\end{proof}

Before proving the main theorem of this subsection, we take the opportunity to prove a variant of Proposition \ref{prop:unstable} in the more general locally compact setting.  This will give us intuition on how to translate the tools of the previous section to the non-discrete setting.

\begin{proposition}\label{prop:cont_amen} Suppose that $G$ is a locally compact, non-compact, amenable group with fixed Haar measure $\mu$. Then $\Th(L^{1}(G),*)$ is unstable.
\end{proposition}
\begin{proof} The proof is similar to the proof of the discrete case (Proposition \ref{prop:unstable}). We recall that a locally compact group is amenable if and only if for every compact subset $K$ of $G$ and $\epsilon > 0$, there exists a compact subset $E$ of $G$ such that $0 < \mu(E) < \infty$ and for every $a \in K$, 
\begin{equation*}
    \frac{\mu(a E \triangle E)}{\mu(E)} < \epsilon,
\end{equation*}
(See \cite[Theorem 7.3 and Proposition 7.11]{Pier1984}). Since $G$ is not compact, for every compact subsets $K$ of $G$, there exists an element $a_{K} \in G$ such that $a_{K} \cdot K \cap K = \emptyset$. Otherwise, we have that for every $c \in G$, $c \cdot K \cap K \neq \emptyset$. Thus there exists $k_1,k_2 \in K$ such that $ck_1 = k_2$ and so $c = k_2 \cdot k^{-1}$. We conclude that $G = K \cdot K^{-1}$. Since group multiplication and inversion are continuous, we conclude that $G$ is compact, a contradiction. 

Let $(\psi_{U})_{U \in \mathcal{O}}$ be an approximate identity. Let $D$ be an ultrafilter on $\mathcal{O}$ such that for every $U \in \mathcal{O}$, $\{X \in \mathcal{O} : X \subseteq U\} \in D$. Consider the ultrapower $\prod_D(L^1(G),*)$. For every $c \in G$, we consider the element $\Psi_{c} = [\psi_{U}(x \cdot c^{-1})]_{D} =[\psi_U\cdot c]_D$. The element $\Psi_{c}$ allows one to translate by $c$ on diagonal elements in the ultrapower. We show that the formula $d(x * y, y)$ is not stable with respect to 
$\operatorname{Th}(L^1(G), *)$. To do so, we construct parameters in the 
ultrapower $\prod_{D} (L^1(G), *)$ that witness the failure of 
Grothendieck's double limit criterion. To construct such parameters, we first build a sequence of compact subsets $(K_n)_{n < \omega}$ and points $(b_n)_{n < \omega}$ in $G$. 

\textbf{Step 0}: Let $K_0$ be a compact set. Let $b_0 = a_{K_0}$. 

\textbf{Step n+1:} Suppose we have constructed $K_i$ and $b_i$ for $0\leq i\leq n$. By amenability, there exists a compact subset $E$ of $G$ such that $0 < \mu(E) < \infty$ and for every $c \in \{b_0,...,b_{n}\}$, 
\begin{equation*}
    \frac{\mu(c E \triangle E )}{\mu(E)} < \frac{1}{n+1}. 
\end{equation*}
We set $K_{n+1} = E$. Now, since any finite union of compact sets remains compact, we may choose $b_{n+1}$ such that 
\begin{equation*}
    \left(b_{n+1} \cdot \bigcup_{m \leq n+1} K_m \right) \cap \bigcup_{m \leq n+1} K_m = \emptyset. 
\end{equation*} 

Thus we have constructed the sequences $(K_n)_{n < \omega}$ and $(b_{n})_{n < \omega}$. We now use these sequences to construct elements in $\prod_{D}(L^{1}(G),*)$. Let $f_n = \Delta\left(\frac{\mathbf{1}_{K_n}}{\mu(K_n)} \right)$ and $g_n = \Psi_{b_n}$, where $\Delta\colon(L^{1}(G),*) \to \prod_{D} (L^{1}(G),*)$ is the standard diagonal embedding. 

Then, by Fact \ref{fact:compute} and Proposition \ref{prop:approx_id} we have
\begin{align*}
    g_n*f_m=\Psi_{b_n}*\Delta\left(\frac{\mathbf{1}_{K_n}}{\mu(K_n)}  \right)&=[(\psi_U\cdot b_{n})*\frac{\mathbf{1}_{K_n}}{\mu(K_n)}]_D=[\psi_U*(b_{n}\cdot \frac{\mathbf{1}_{K_n}}{\mu(K_n)})]_D\\
    &=[b_{n}\cdot\frac{\mathbf{1}_{K_n}}{\mu(K_n)}]_D=\Delta\left(b_{n}\cdot \frac{\mathbf{1}_{K_n}}{\mu(K_n)}\right)=\Delta\left(\frac{\mathbf{1}_{b_{n} K_n}}{\mu(K_n)}\right).
\end{align*}
Let $\varphi(x,y) := d(x * y, y)$, then 
\begin{equation*}
    \lim_{n \to \infty} \lim_{m \to \infty} \varphi(g_n , f_m) = 0 \text{ while } \lim_{m \to \infty} \lim_{n \to \infty} \varphi(g_n,f_m) = 2. 
\end{equation*}
Thus $\varphi(x,y)$ violates Grothendieck's double limit criterion and so $\Th(L^{1}(G),*)$ is unstable. 
\end{proof}

While the previous proposition is a direct generalization of Proposition \ref{prop:unstable}, we now show that $(L^{1}(G),*)$ is unstable whenever $G$ is non-discrete. We will use the following fact. 

\begin{fact}\label{fact:nondiscrete} Suppose that $G$ is a non-discrete locally compact group with fixed Haar measure $\mu$. Then for every $\epsilon > 0$ there exists some symmetric open neighborhood of the identity $U$ such that $0 < \mu(U) < \epsilon$. 
\end{fact}

\begin{proof} This is classical. Since $G$ is not discrete, $\mu(\{e\}) = 0$ (e.g., \cite[Proposition 1.4.4]{DeitmarEchterhoff2014}). Since $\mu$ is a Haar measure, it is outer regular, and thus $\mu(\{e\}) = \inf\{\mu(U): U$ is an open neighborhood of $e\}$ (e.g., \cite[Theorem 2.10]{folland2016harmonic}). Since every non-empty open set have positive Haar measure, we may find $U$ such that $0 < \mu(U) < \epsilon$. To ensure that $U$ is symmetric, replace $U$ with $U \cap U^{-1}$ if necessary. 
\end{proof}

\begin{proposition}\label{notstab-compact} Suppose that $G$ is a locally compact, non-discrete group with fixed Haar measure $\mu$. Then the formula $\varphi(x,y)=\lVert x-x*y\lVert=d(x,x*y)$
is unstable with respect to $\Th(L^1(G),*)$.
\end{proposition}

\begin{proof}
The key ingredient in this proof is the existence of a non-trivial bounded approximate identity. Without loss of generality, let $\mathcal{O}$ be the family of symmetric open neighborhoods of the identity (with finite Haar measure), consider the bounded approximate identity given by $(\frac{1}{\mu(V)}\mathbf{1}_{V})_{V \in \mathcal{O}}$. For notational convenience, we write $\theta_{V} \coloneqq \frac{1}{\mu(V)}\mathbf{1}_{V}$. We selectively construct inductively a sequence of neighborhoods $\{V_n\}_{n < \omega}$ around the identity. First, we let $V_0$ be any neighborhood of the identity such that $ 0 < \mu(V_0) < \infty$. Assuming that we are given  $\{V_i\}_{i\leq n}$, we define a symmetric open set $V_{n+1}$ such that $V_{n+1} \subseteq \bigcap_{i \leq n} V_i$, $\mu(V_{n+1}) < \mu(V_{n})/2$ and $\lVert \theta_{V_{n+1}}*f-f\lVert<\frac{1}{n+1}$ for $f=\theta_{V_i}$ with $i\leq n$. Such a sequence of open sets exists since $(\theta_{V})_{V \in \mathcal{O}}$ is a b.a.i. (see Proposition \ref{prop:approx_id} and Fact \ref{fact:nondiscrete}). 

Let $f_n=\theta_{V_n}$ and note that $\lVert f_n\lVert=1$. Now let $a_n=f_n$, $b_m=f_m$. Then by the definition of the sequence $\{V_n\}_{n < \omega}$
 \begin{equation*}
     \lim_{m\to \infty}\lVert a_n - a_n*b_m\lVert=0,
 \end{equation*}
and thus 
\begin{equation*}
    \lim_{n \to \infty} \lim_{m\to \infty}\lVert a_n-a_n*b_m\lVert=0. 
\end{equation*}
On the other hand, notice that if we fix $m$ and consider $n > m$,
\begin{equation*}
    \lVert a_n - a_n * b_m \rVert \approx_{\frac{1}{n}} \lVert a_n - b_m \rVert = 2 \left(1 - \frac{\mu(V_n)}{\mu(V_m)} \right) \geq  2 \left(1 - \frac{1}{2^{n-m}} \right). 
\end{equation*}

Hence, for any fixed $m$, 
\begin{equation*}
    \lim_{n\to \infty} \lVert a_n - a_n * b_m \rVert = 2, 
\end{equation*}
and so, 
\begin{equation*}
    \lim_{m \to \infty} \lim_{n\to \infty} \lVert a_n - a_n * b_m \rVert = 2. 
\end{equation*}

Thus $\varphi(x,y)$ violates Grothendieck's double limit criterion and so $\Th(L^{1}(G),*)$ is unstable. 
\end{proof}

Let us point out some connections between this result and abstract harmonic analysis. We actually proved a stronger statement, the function $d(x*y,y)$ is not WAP, that is, there are sequences in $\{a_n\}_n$, $\{b_m\}_m$ in the unit ball of $(L^1(G),*)$ (without the need of going to an elementary extension) such that 
\begin{equation*}
    2=\lim_{n \to \infty} \lim_{m \to \infty}  d(a_n*b_m,b_m)\neq \lim_{m \to \infty} \lim_{n \to \infty} d(a_n*b_m,b_m)=0. 
\end{equation*}

\begin{theorem}\label{thm:ContIP} Let $G$ be a locally compact group with fixed Haar measure $\mu$ and suppose that for every positive integer $n$, there exists an element of order greater than $n$. Then $d(x * y, z)\dotdiv\frac{1}{2}$ has $\mathrm{TP}_2$ and $d(x * y, y)$ has $\mathrm{IP}$ with respect to $\Th(L^{1}(G),*)$. 
\end{theorem}

\begin{proof}  Fix $k \in \mathbb{N}$. Let $(\psi_{U})_{U \in \mathcal{O}}$ be an approximate identity for $L^{1}(G)$. Let $D$ be an ultrafilter on $\mathcal{O}$ such that for every $U \in \mathcal{O}$, the set $\{ X \in \mathcal{O}: X \subseteq U\} \in D$. Consider the ultrapower $\prod_D(L^1(G),*)$. Let $\Psi = [\psi_{U}]_{D}$. For every $a \in G$, we let $\Psi_{a} = [\psi_{U}( x \cdot a^{-1})]_{D}$. The element $\Psi_{a}$ allows one to translate by $a$ on diagonal elements in the ultrapower. More formally, if $\Delta\colon (L^{1}(G),*) \to \prod_{D} (L^{1}(G),*)$ is the standard diagonal embedding, then for every $a \in G$ and $f \in L^{1}(G)$, 
$$\Psi_{a} * \Delta(f) = [\psi_{U}(x \cdot a^{-1}) * f]_{D} =[\psi_{U} * (a \cdot f)]_{D} =  \Delta(a \cdot f).$$

The second equality follows from Fact \ref{fact:compute}. We now justify the third equality. Indeed, the following equivalence follows directly from the definition of a metric ultrapower:
\begin{equation*}
    [\psi_{U} * (a \cdot f)]_D = \Delta(a \cdot f) \Longleftrightarrow \forall \epsilon > 0, \{X \in \mathcal{O} : \lVert \psi_{X} * (a \cdot f) - a \cdot f \rVert_{L^{1}(G)} \leq \epsilon \} \in D.  
\end{equation*}
We argue that the right-hand-side of the equivalence is satisfied. Since $(\psi_{U})_{U \in \mathcal{O}}$ is an approximate identity, it follows from Proposition \ref{prop:approx_id} that for every $\epsilon > 0$, there exists some $U_{\epsilon} \in \mathcal{O}$ such that for any $X \in \mathcal{O}$ with $X \subseteq U_{\epsilon}$, $\lVert \psi_{X} *( a \cdot f) - a \cdot f \rVert_{L^{1}(G)} \leq \epsilon$. By construction of our ultrafilter $D$, for every $\epsilon > 0$, we have that $\{X \in \mathcal{O}: X \subseteq U_{\epsilon}\} \in D$. Since ultrafilters are closed under supersets, 
\begin{equation*}
    \{X \in \mathcal{O}: X \subseteq U_{\epsilon}\} \subseteq \{X \in \mathcal{O}: \lVert \psi_{X} * (a \cdot f) - a \cdot f \rVert_{L^{1}(G)} \leq \epsilon\} \in D. 
\end{equation*}
Thus, the right-hand-side of the equivalence is satisfied, justifying the third equality.

Choose $k$ distinct primes $p_1,...,p_k$ which are all larger than $k$. Let $n = \prod_{i = 1}^{k} p_k$. Let $m \gg n$. For each $i \leq k$, we let $m_i$ denote the set $p_i\mathbb{Z} \cap [0,m]$. Choose $a_{m} \in G$ with order larger than $2m$. Then the elements $\{a_{m}, a_{m}^{2},..., a_{m}^{2m}\}$ are all distinct. By Fact \ref{fact:translate}, there exists some open neighborhood of the identity $V_{m}$ such that $a_{m}^{w_1} \cdot V_{m} \cap a_{m}^{w_2} \cdot V_{m} = \emptyset$ for $w_1,w_2 \leq 2m$. Since $V$ is open, we have that $\mu(V) > 0$ and since $\mu$ is the Haar measure, it is invariant under translations, and so $\mu(h V) > 0$ for every $h \in G$.

Then, for every $m < \omega$ and $i \leq k$, we define,
\begin{equation*}
    f_{i,m}(x) = \frac{1}{|m_i| \cdot  \mu(V_{m})} \sum_{\substack{r \in [0,m] \\ r \in p_i\mathbb{Z}}} \mathbf{1}_{a_{m}^{r}V_{m}}(x),
\end{equation*}  
and the translates (for $j \leq k$); 
\begin{equation*}
        g_{i,j,m}(x) = \frac{1}{|m_i| \cdot  \mu(V_{m})} \sum_{\substack{r \in [0,m] \\ r \in p_i\mathbb{Z}}} \mathbf{1}_{a_{m}^{r + j}V_{m}}(x),
\end{equation*}
Note that for $j_1<j_2\leq k$ the functions  $g_{i,j_1,m}(x)$ and $g_{i,j_2,m}(x)$ have disjoint support and thus $d(g_{i,j_1,m},g_{i,j_2,m})=2$.

We identify $f_{i,m}$ and $g_{i,j,m}$ with their images under the diagonal embedding in $\prod_{D} L^{1}(G)$. Let $E$ be a non-principal ultrafilter on $\mathbb{N}$. Taking an ultrapower again, we obtain a new structure $N = \prod_{E} (\prod_{D} L^{1}(G))$ and norm one elements $f_{i} = [f_{i,m}]_{E}$ and  $g_{i,j} = [g_{i,j,m}]_{E}$, again identifying the functions with their diagonal embedding. We first argue that paths are consistent. Fix $\eta\colon [k] \to [k]$. By the Chinese remainder theorem, there exists some integer $t_{\eta} \leq n$ (recall that $n = \prod_{i = 1}^{k} p_k$) such that $t_{\eta} \equiv \eta(i) \mod p_{i}$ for each $i \leq k$. We claim that for each $i \leq k$, $d(\Psi_{t_{\eta}} * f_{i}, g_{i,\eta(i)}) = 0$, where $\Psi_{t_{\eta}}$ is a short form of $[\Psi_{a_m^{t_\eta}}]_{E}$. Indeed, since $\eta(i) < k < p_i$, we have that $t_{\eta} = bp_i + \eta(i)$ for some non-negative integer $b \leq n$. Thus, 
\begin{align*}
    &d(\Psi_{t_\eta} * f_i, g_{i,\eta(i)}) = \lim_{m \to \infty} d(\Psi_{t_{\eta}} * f_{i,m}, g_{i,\eta(i),m}) \\
    &= \lim_{m \to \infty} \frac{1}{|m_i| \cdot \mu(V_m)} \int_{G} \left| \sum_{\substack{r \in [0,m] \\ r \in p_i \mathbb{Z}}} \mathbf{1}_{a_m^{r + t_{\eta} } V_{m}}(x) - \sum_{\substack{r \in [0,m] \\ r \in p_i \mathbb{Z}}}\mathbf{1}_{a_{m}^{r + \eta(i)} V_{m}}(x)   \right| d\mu\\ 
    &= \lim_{m \to \infty} \frac{1}{|m_i| \cdot \mu(V_m)} \int_{G} \sum_{\substack{r \in [0,bp_i] \cup [m,m + bp_i] \\ r \in p_i \mathbb{Z}}}\mathbf{1}_{a_{m}^{r+\eta(i)} V_{m}}(x) d\mu \\
    &\leq \lim_{m \to \infty} \frac{2b}{|m_i|} = 0. 
\end{align*}
For the final equality, recall that $m \gg n$ and thus $|m_i| \gg n\geq b$.

Now we argue that the rows are $2$-inconsistent. Notice that $d(g_{i,j},g_{i,j'}) = 2$ whenever $j \neq j'$. Hence, if there exists some $h$ such that 
\begin{equation*}
    d(h * f_{i},g_{i,j}) \dotdiv \frac{1}{2}=0 \text{ and } d(h * f_{i},g_{i,j'}) \dotdiv \frac{1}{2}=0,
\end{equation*}
then by the triangle inequality, we have that $d(g_{i,j},g_{i,j'}) \leq 1$, a contradiction. 

By an argument similar to the proof of Theorem \ref{prop:general}, one can show that for every $n\geq 1$, the formula $d(x * y, y)$ shatters $\{f_{i}: i \leq n\}$ and thus it has $\mathrm{IP}$. 
\end{proof}

\section{Embedding abelian Lie groups}

In this section, we prove that certain continuous convolution structures can be approximated by discrete ones. At the same time, the results of this section show that the metric ultraproducts of discrete convolution algebras are complicated enough to encode their continuous counterparts. 

Fix a connected abelian Lie group $G$ with fixed Haar measure $\mu$ and bi-invariant metric $d$. We will show that there exists a sequence of finite abelian groups $(H_i)_{i < \omega}$ such that $(L^{1}(G),*)$ metrically embeds into $\prod_{D} (\ell^{1}(H_i),*)$. Moreover, we can take this sequence of finite groups to be of the form $(\mathbb{Z}/n_i\mathbb{Z})_{i < \omega}$. Our first fact is just the standard classification of connected abelian Lie groups (see e.g., \cite[Section 4.2, Page 87]{PC}). 

\begin{fact}\label{fact:classification} Suppose that $G$ is a connected abelian Lie group. Then $G \cong \mathbb{R}^{m} \times \mathbb{T}^{k}$ for some $m,k \geq 0$, where $\mathbb{T}^{k} = (S^{1})^{k}$ is the $k$-dimensional torus. 
\end{fact}

\begin{definition} For our purposes, a \textbf{lattice} $\Lambda$ is a discrete subgroup of $G$ such that $G/\Lambda$ is compact. A \textbf{fundamental domain} with respect to a lattice $\Lambda$ is a subset $F$ of $G$ such that 
\begin{enumerate}
    \item $\bigcup_{\lambda \in \Lambda} (\lambda F) = G$
    \item For any distinct $\lambda_1,\lambda_2 \in G$, the interior of $\lambda_1F \cap \lambda_2 F$ is empty. 
\end{enumerate}
\end{definition}

\begin{definition}\label{defn:admissible} Let $G$ be a connected abelian Lie group. We say that a sequence of lattices paired with fundamental domains $(\Lambda_n,F_n)_{n < \omega}$ is \textbf{admissible} (for $G$) if
\begin{enumerate}
    \item For each $n < \omega$, $\{\lambda F_n: \lambda \in \Lambda_n\}$ forms a disjoint partition of $G$. Thus, for each element $a \in G$, we let $\lambda_{n,a}$ be the unique element of $\Lambda_n$ such that $a \in \lambda_{n,a}F_n$.
    \item For each $n < \omega$, $\mu(F_{n} \triangle F_{n}^{-1}) = 0$.
    \item For every $\epsilon > 0$, there exists $n < \omega$ such that for any $m > n$ and any $a \in G$, $d(a,\lambda_{m,a}) < \epsilon$.
\end{enumerate}
\end{definition}

\begin{fact} If $G$ is a connected abelian Lie group, then there exists a sequence of lattices and fundamental domains such that $(\lambda_n,F_n)_{n < \omega}$ is admissible for $G$.
\end{fact}

\begin{proof}
If $G \cong \mathbb{R}^m \times \mathbb{T}^k$, for each $0< n < \omega$ define
\[
\Lambda_n = \left(\frac{1}{n}\mathbb{Z}\right)^m \times (C_n)^k,
\]
which we identify with its canonical copy inside $G$, where $C_n$ is the cyclic group of $n$-th roots of unity, and
\[
F_n = \left[-\frac{1}{2n}, \frac{1}{2n}\right)^m \times \left( \left\{ e^{2\pi i \theta} : \theta \in \left[-\frac{1}{2n}, \frac{1}{2n}\right) \right\} \right)^k.
\]
Then each $F_n$ is a fundamental domain for $\Lambda_n$, and $(\Lambda_n, F_n)_{0< n < \omega} $ is admissible.
\end{proof}

\begin{example} Let $G = \mathbb{R}^{2}$. The sequence of lattices $(\frac{1}{n}\mathbb{Z} \times \mathbb{Z})_{ 0< n < \omega}$ paired with any fundamental domains is not admissible, while the sequence of lattices $(\frac{1}{n}\mathbb{Z} \times \frac{1}{n}\mathbb{Z})_{ 0< n < \omega}$ paired with $((\frac{1}{2n},\frac{1}{2n}]^2)_{0< n < \omega}$ is admissible. 
\end{example}

\begin{lemma}\label{lemma:approx} Let $G$ be a connected abelian Lie group, $(\Lambda_n,F_n)_{n < \omega}$ be an admissible sequence, and $g \in C_{c}(G)$. For every $\epsilon > 0$ there exists $n < \omega$ and a positive real number $r$ such that for every $m > n$, 
\begin{equation*} 
\sup_{a \in G} |g(a) - g_{m,r}(a)| < \epsilon, 
\end{equation*} 
where $g_{m,r}(x) = \sum_{\lambda \in \Lambda_m \cap B_{r}} g(\lambda) \cdot \mathbf{1}_{\lambda F_m}(x)$ and $B_{r}$ is the closed ball of radius $r$ centered at the identity. 
\end{lemma}

\begin{proof} Fix $\epsilon > 0$. Since $g$ has compact support, we may choose $r$ such that $B_{r}$ contains the support of $g$. Moreover, $g$ is uniformly continuous. Thus, choose $\delta > 0$ such that whenever $d(x,y) < \delta$, it follows that $d(g(x),g(y)) < \epsilon$. Since $(\Lambda_n,F_n)_{n < \omega}$ is admissible, we can choose $i$ such that for every $j > i$ and $a \in G$, $d(a,\lambda_{j,a}) < \delta$. Setting $n = i$, we notice that for $m > n$,
\begin{equation*} 
d(g(a),g_{m,r}(a)) \leq d(g(a),g(\lambda_{m,a})) + d(g(\lambda_{m,a}),g_{m,r}(a)) < \epsilon + 0. \qedhere
\end{equation*} 
Note that $g_{m,r}(a) = \sum_{\lambda \in \Lambda_m \cap B_{r}} g(\lambda) \cdot \mathbf{1}_{\lambda F_m}(a) = g(\lambda_{m,a})$. 
\end{proof} 

A proof of the following fact can be found in \cite[Corollary 8.A.22]{cornulier2016metric}.

\begin{fact}\label{fact:FG} Let $L$ be a connected Lie group. Then every discrete normal subgroup of $L$ is finitely generated. 
\end{fact}

We now prove the main theorem of this section. 

\begin{theorem}\label{theorem:embeds} Let $G$ be a connected abelian Lie group. Then there exists a sequence of finite abelian groups $(H_n)_{n < \omega}$ and an ultrafilter $D$ on $\omega$ such that $(L^1(G),*)$ metrically embeds into $\prod_{D}(\ell^1(H_n),*)$. 
\end{theorem}

\begin{proof}
Let $(\Lambda_n,F_n)_{n < \omega}$ be a sequence of lattices and fundamental domains which is admissible for $G$. We recall that $C_{c}(G) \subseteq_{dense} L^{1}(G)$ and since metric ultraproducts are complete, it suffices to prove that $C_{c}(G)$ metrically embeds into the ultraproduct with the $L^1(G)$ norm. 

    For any positive real number $r$ we let $B_{r}$ be the closed ball of radius $r$ centered at the identity. For each $n < \omega$, the set $\Lambda_n \cap B_n$ is finite since $B_n$ is compact and $\Lambda_n$ is discrete. Since each $\Lambda_n$ is a lattice in a connected abelian Lie group, $\Lambda_n$ is isomorphic to the product of finitely many copies of $\mathbb{Z}$ and a finite abelian group (by the fundamental theory of finitely generated abelian groups and Fact \ref{fact:FG}). Thus, it is residually finite, and so we may choose a finite group $H_{n}$ and a surjective homomorphism $\rho_n \colon \Lambda_n \to H_n$ such that the map $\rho_n|_{\Lambda_n \cap B_n}$ is injective. Since $G$ is abelian, so is $\Lambda_n$. Since the map $\rho_n$ is surjective, the group $H_n$ is also abelian.

    Consider the family of finite groups given by $(H_n)_{n < \omega}$. Equip each $H_n$ with the measure $\mu_n$ where for each $s \in H_{n}$, $\mu_n(\{s\}) = \mu(F_n)$. Notice that if $f,g \in \ell^{1}(H_n)$, then 

    \begin{equation*}
        (f * g)(\{t\}) = \sum_{s \in H_n} f(s) g(s^{-1} t) \mu_n(\{s\}) = \sum_{s \in H_n} f(s) g(s^{-1} t) \mu(F_n).
    \end{equation*}

Let $D$ be a non-principal ultrafilter over $\omega$. We define the map $\Phi\colon C_{c}(G) \to \prod_{D} \ell^{1}(H_n)$ via $\Phi(f) = [\tilde{f}_n]_{D}$ as follows: let $C$ be the support of $f$. If $C \not \subseteq B_n$, then let $\tilde{f}_n$ be any element of $\ell^{1}(H_n)$. If $C \subseteq B_n$, then
\begin{equation*}
    \tilde{f}_{n}(t) = 
    \begin{cases}
        f(a) & \text{if there exists } a \in \Lambda_n \cap B_n \text{ such that } \rho_n(a) = t, \\
        0 & \text{otherwise}.
    \end{cases}
\end{equation*}
Notice that $f_n$ is well-defined since $\rho_n$ is injective on $\Lambda_n \cap B_n$. We claim that $\Phi$ extends to a metric embedding of $(L^{1}(G),*)$ into $\prod_{D}(\ell^{1}(H_n),*)$. We check that $\Phi$ preserves the norm and convolution and leave the other operations as an exercise for the reader. Note that, in particular, the unit ball of $(L^{1}(G),*)$ goes into the unit ball of 
$\prod_{D}(\ell^{1}(H_n),*)$.

\begin{claim-star}The map $\Phi$ preserves norm. 
\end{claim-star}
\begin{claimproof} Fix $f \in L^{1}(G)$. It suffices to prove that
\begin{equation*}
    \lim_{n \to \infty} \lVert \tilde{f}_n \rVert_{\ell^{1}(H)}= \lVert f \rVert_{L^{1}(G)}. 
\end{equation*}

Fix $\epsilon > 0$. By Lemma \ref{lemma:approx}, there exists some integer $N_\epsilon$ such that for every $m > N_{\epsilon}$, $\lVert f - f_{m,m} \rVert_\infty < \epsilon$ where $f_{m,m} = \sum_{\lambda \in \Lambda_{m} \cap B_{m}} f(\lambda) \cdot \mathbf{1}_{\lambda F_{m}}(x)$. For $m > N_{\epsilon}$,  
\begin{align*}
    \lVert \tilde{f}_m \rVert_{\ell^1(H_m)} &= \sum_{s \in H_m} |\tilde{f}_m(s)| \, \mu_m(\{s\}) \\ 
    &= \sum_{s \in H_m} |\tilde{f}_m(s)| \, \mu(F_m) \\ 
    &= \sum_{\lambda \in \Lambda_m \cap B_m } |f(\lambda)| \, \mu(F_m) \\
    &= \sum_{\lambda \in \Lambda_m \cap B_m} \int_{G} |f(\lambda)| \mathbf{1}_{\lambda F_m}(x) \, d\mu(x) \\ 
    &= \lVert f_{m,m} \rVert_{L^1(G)} \approx_{\epsilon} \lVert f \rVert_{L^1(G)}
\end{align*}
Hence, the claim holds. 
\end{claimproof}

\begin{claim-star} The map $\Phi$ preserves convolution. 
\end{claim-star}

\begin{claimproof}Fix $f, g \in C_{c}(G)$ with compact supports $C$ and $D$ respectively. Let $h = f * g$. Notice that $h$ again has compact support and in particular, the support of $h$ is a subset of $C \cdot D$. It suffices to show that 
\begin{equation*}
    \lim_{n \to \infty} \lVert \tilde{h}_n - \tilde{f}_n * \tilde{g}_n \lVert_{\ell^{1}(H_n)} = 0. 
\end{equation*}
Fix $1 > \epsilon > 0$. By Lemma \ref{lemma:approx}, there exists some $n$ such that 
\begin{enumerate}
    \item $C \cup D \cup C \cdot D \subseteq B_{n}$,
    \item For $k \in \{h,f,g\}$ and for $m > n$, 
    \begin{equation*}
        \sup_{x \in G}|k(x) - \sum_{a \in \Lambda_m \cap B_m} k(a) \cdot \mathbf{1}_{aF_m}(x)| < \delta \coloneqq \frac{\epsilon}{\lVert f \rVert_{\infty} + \lVert g \rVert_{\infty} + 1 }. 
    \end{equation*}
\end{enumerate}
Fix $m > n$. We now argue that for every $t \in H_{m}$,
\begin{equation*}
    |\tilde{h}_m(t) - (\tilde{f}_{m}* \tilde{g}_m)(t)| < \epsilon. 
\end{equation*}
We have two cases.

\textbf{Case 1:} There does not exist some $a \in \Lambda_{m} \cap B_m$ such that $\rho_{m}(a) = t$. Then, $\tilde{h}_{m}(t) = 0$ by definition. On the other hand, we have that 
\begin{align*}
    (\tilde{f}_{m} * \tilde{g}_{m})(t) &= \sum_{s \in H_m} \tilde{f}_{m}(s) \tilde{g}_m(s^{-1} \cdot t) \mu_m(\{s\}) \\
    &\overset{(1)}{=} \sum_{s \in \im(\rho_m|_{\Lambda_{m}\cap B_{m}})} \tilde{f}_{m}(s) \tilde{g}_m(s^{-1} \cdot t)\mu_m(\{s\}) \overset{(2)}{=} 0. 
\end{align*}

We briefly justify the sequence of equations above: 
\begin{enumerate}
    \item $\tilde{f}_{m}(s) > 0$ only if $s \in \im(\rho_m|_{\Lambda_{m} \cap B_{m}})$. 
    \item Towards a contradiction, suppose that there exists some $s \in \im(\rho_{m}|_{\Lambda_{m} \cap B_{m}})$ such that $|\tilde{f}_{m}(s) \tilde{g}_m(s^{-1} \cdot t)| > 0$. Then $|\tilde{f}_{m}(s)| > 0$ and $|\tilde{g}_{m}(s^{-1} \cdot t)| > 0$. By construction, there exists some $a \in \supp(f)$ and $b \in \supp(g)$ such that $\rho_{m}(a) = s$ and $\rho_{m}(b) = s^{-1} \cdot t$. Then $a \cdot b \in C \cdot D \subseteq B_{m}$. Thus $\rho_{m}(a \cdot b) = \rho_{m}(a) \cdot \rho_{m}(b) = s \cdot s^{-1} \cdot t = t$, which contradicts our initial assumption. 
\end{enumerate}

\textbf{Case 2:} There exists some $a \in \Lambda_{m} \cap B_m$ such that $\rho_{m}(a) = t$. Consider the following computation:

\begin{align*}
\tilde{h}_{m}(t) &= (f * g)(a) \overset{(1)}{\approx_{\epsilon}} \left( \left( \sum_{c \in \Lambda_m \cap B_m} f(c) \mathbf{1}_{cF_m} \right) * \left( \sum_{b \in \Lambda_m \cap B_m} g(b) \mathbf{1}_{bF_m} \right)\right)(a) \\
&= \sum_{c \in \Lambda_m \cap B_m} \sum_{b \in \Lambda_m \cap B_m} \int_{x \in G} f(c)g(b)\mathbf{1}_{cF_m}(x) \cdot \mathbf{1}_{bF_m}(x^{-1} \cdot a) d\mu \\ 
&\overset{(2)}{=} \sum_{c \in \Lambda_m \cap B_m} \int_{x \in G} f(c)g(c^{-1} \cdot a) \mathbf{1}_{cF_m}(x) d\mu \\
&= \sum_{c \in \Lambda_m \cap B_m}  f(c)g(c^{-1} \cdot a) \mu(F_m) \\
&\overset{(3)}{=} \sum_{s \in H_m} \tilde{f}_m(s)\tilde{g}_m(s^{-1} \cdot t) \mu_m(\{s\}) \\
&=(\tilde{f}_m * \tilde{g}_m)(t). 
\end{align*}

We provide several justifications: 

\begin{enumerate}
    \item Since $\lVert f - f_{m,m} \rVert_{\infty}, \lVert g - g_{m,m} \rVert_{\infty} < \delta$, it follows that 
    \begin{align*}
        \lVert f * g - f_{m,m} * g_{m,m} \rVert_{\infty} &< \delta \lVert g \rVert_{\infty} + \delta \lVert f_{m,m} \rVert_{\infty} \\ 
        &< \epsilon \left( \frac{\lVert g \rVert_{\infty} + \lVert f_{m,m} \rVert_{\infty}}{\lVert g \rVert_{\infty} + \lVert f \rVert_{\infty} + 1} \right) < \epsilon. 
    \end{align*}
\item Notice that for every $x \in G$, 
    \begin{equation*}
        \mathbf{1}_{cF_m}(x) \cdot \mathbf{1}_{bF_m}(x^{-1} \cdot a) = 1
    \end{equation*}
    if and only if $x \in cF_m$ and $x^{-1} \cdot a \in bF_m$. Now, 
    \begin{align*}
        x^{-1} \cdot a \in bF_m &\Longleftrightarrow x^{-1} \in bF_m a^{-1} \\
        &\Longleftrightarrow x \in a F_m^{-1} b^{-1} \\
        &\Longleftrightarrow  x \in ab^{-1} F_m^{-1},
    \end{align*} 
The last bi-implication follows from the fact that $G$ is abelian. Since $\mu(F_m \triangle F_m^{-1}) = 0$, 
    \begin{align*}
        \int_{x \in G} \mathbf{1}_{cF_m}(x) \cdot \mathbf{1}_{bF_m}(x^{-1} \cdot a) d\mu &= \int_{x \in G} \mathbf{1}_{cF_m}(x) \cdot \mathbf{1}_{ab^{-1}F_m^{-1}}(x) d\mu \\
        &= \int_{x \in G} \mathbf{1}_{cF_m}(x) \cdot \mathbf{1}_{ab^{-1}F_m}(x) d\mu 
    \end{align*}
    Notice that the final term is positive only if $ab^{-1} = c$ (note that here we used the uniqueness condition in part (1) of Definition \ref{defn:admissible}) or in other words, $b = c^{-1} a$. Thus the change in parameters is justified. 
    \item Notice that 
    \begin{align*}
        \sum_{c \in \Lambda_m \cap B_m} f(c)g(c^{-1} \cdot a) \mu(F_m) = \sum_{\substack{c \in \Lambda_m \cap B_m \\ |f(c)g(c^{-1} a)| > 0}} f(c)g(c^{-1} \cdot a) \mu(F_m)
    \end{align*}
    If $|f(c)g(c^{-1} \cdot a)| > 0$, then $c \in \supp(f)$ and $c^{-1} a \in \supp(g)$. By assumption, we have that $\supp(f), \supp(g) \subseteq B_n \subset B_m$. Moreover, since $a$ was assumed to be in $\Lambda_m \cap B_m$, we conclude that if $|f(c)g(c^{-1} \cdot a)| > 0$, then $c, c^{-1} \cdot a \in \Lambda_m \cap B_m$. If $|f(c)g(c^{-1} \cdot a)| > 0$, then $f(c) = \tilde{f}_{m}(\rho_m(c))$ and 
\begin{equation*}
        g(c^{-1} \cdot a) = \tilde{g}_m(\rho_m(c^{-1} \cdot a)) =  \tilde{g}_m(\rho_m(c^{-1}) \cdot \rho_m(a)) = \tilde{g}_m(\rho_m(c)^{-1} \cdot t).
\end{equation*}
    Hence, 
    \begin{align*}
         &\sum_{\substack{c \in \Lambda_m \cap B_m \\ |f(c)g(c^{-1} a)| > 0}} f(c)g(c^{-1} \cdot a) \mu(F_m) \\  &=\sum_{\substack{c \in \Lambda_m \cap B_m \\ |f(c) g(c^{-1} a)| > 0}} \tilde{f}_m(\rho_m(c))\tilde{g}_m(\rho_m(c)^{-1} \cdot t) \mu(F_m) \\
         &= \sum_{s \in A_{m,t}} \tilde{f}_m(s) \tilde{g}_{m}(s^{-1} \cdot t) \mu_{m}(\{s\}) \\
         &= \sum_{s \in H_m} \tilde{f}_{m}(s) \tilde{g}_m(s^{-1} \cdot t)\mu_{m}(\{s\}), 
    \end{align*}
    where $A_{m,t} = \{s \in H_{m}: \exists c \in \Lambda_{m} \cap B_{m} $ such that $\rho_{m}(c) = s$ and $|f(c) \cdot g(c^{-1}a)| >0\}$. Notice that if $s \in H_m \backslash A_{m,t}$, then suppose there does not exist some $c \in \Lambda_m \cap B_m$ such that $\rho_{m}(c) = s$. In this case, we have that $\tilde{f}_{m}(s) = 0$ by construction and thus $\tilde{f}_{m}(s) \cdot \tilde{g}_{m}(s^{-1} \cdot t) =0$. Hence, adding this term does not change the value of the summation. Otherwise, if there does exist some $c \in \Lambda_{m} \cap B_{m}$ such that $\rho_{m}(c) = s$ but $f(c) \cdot g(c^{-1} \cdot a) = 0$, then $\tilde{f}_m(s)\tilde{g}_m(s^{-1} \cdot t ) = 0$. Again, we may add this term to the summation without effecting its value. 
\end{enumerate}
\end{claimproof}
As we said earlier, it is straightforward to check that the map $\Phi$ preserves the other operations, i.e.,  addition, min, max, and so we leave these as an exercise to the reader. 
\end{proof}

\begin{corollary} Let $G$ be a connected abelian Lie group. Then there exists a sequence of natural numbers $(n_i)_{i < \omega}$ such that $L^1(G)$ isometrically embeds into $\prod_{D}\ell^1(\mathbb{Z}/n_i\mathbb{Z})$.
\end{corollary}

\begin{proof} Suppose that $G \cong \mathbb{R}^{m} \times \mathbb{T}^{k}$. In the proof of Theorem \ref{theorem:embeds}, one can find an indexed family of distinct primes $\{p_{(i,j)} : i < \omega, j \leq m + k\}$ and choose $H_i$ so that,
\begin{equation*}
    H_i \cong \prod_{j = 1}^{m+k} \mathbb{Z}/p_{(i,j)}\mathbb{Z}. 
\end{equation*}
The statement thus holds by the fundamental theorem of finitely generated abelian groups, i.e., for each $i < \omega$, $\prod_{j =1}^{m + k} \mathbb{Z}/p_{(i,j)}\mathbb{Z} \cong  \mathbb{Z}/\prod_{j =1}^{m + k}p_{(i,j)}\mathbb{Z}$. 
\end{proof}

\begin{question}
For which non-abelian Lie group $G$ does there exist an embedding $\Psi\colon (L^{1}(G),*) \to \prod_{D} (\ell^{1}(G_i),*)$ where $(G_i)_{i \in I}$ is a family of finite groups and $D$ is an ultrafilter over $I$? Does it follow if $G$ is a Lie group with a family of admissible lattices and fundamental domains which are residually finite?
\end{question}

\printbibliography

\end{document}